\documentclass[11pt,reqno]{amsart}

\usepackage{amsmath,amssymb,mathrsfs,amsthm}
\usepackage{graphicx,cite,cases,enumitem}

\usepackage{arydshln}

\newtheorem{theorem}{Theorem}[section]

\newtheorem{lemma}[theorem]{Lemma}

\newtheorem{definition}[theorem]{Definition}

\numberwithin{equation}{section}

\begin{document}

\title[analysis of acoustic scattering by an elastic
obstacle]{Analysis of transient acoustic scattering by an elastic obstacle}%

\author{Peijun Li}
\address{Department of Mathematics, Purdue University, West Lafayette, Indiana
47907, USA.}
\email{lipeijun@math.purdue.edu}

\author{Lei Zhang}
\address{School of Mathematics, Heilongjiang University, Harbin 150080,
PRC. Heilongjiang Provincial Key Laboratory of Complex Systems Theory and
Computation, Harbin 150080, PRC.}
\email{zl19802003@163.com}

\thanks{The research of LZ was supported in part by a National Natural
Science Foundation of China (No. 11871198, No. 11801116) and the Special Funds
of Science and Technology Innovation Talents of Harbin (No. 2017RAQXJ099).}



\begin{abstract}
Consider the scattering of an acoustic plane wave by a bounded elastic obstacle
which is immersed in an open space filled with a homogeneous medium. This paper
concerns the mathematical analysis of the coupled two- and three-dimensional
acoustic-elastic wave propagation problem in the time-domain. A compressed
coordinate transformation is proposed to reduce equivalently the scattering
problem into an initial-boundary value problem in a bounded domain over a finite
time interval. The reduced problem is shown to have a unique weak solution by
using the Galerkin method. The stability estimate and an a priori estimate with
explicit time dependence are obtained for the weak solution. The reduced model
problem is suitable for numerical simulations. The proposed method is
applicable to many other time-domain scattering problems imposed in open
domains.
\end{abstract}

\maketitle

\section{introduction}

Consider the scattering of a time-domain acoustic plane wave by a bounded and
penetrable elastic obstacle which is immersed in an open space filled with a
homogeneous acoustic medium such as air or fluid. The obstacle is assumed to be
made of a homogeneous and isotropic elastic medium. When the incident wave hits
on the surface of the obstacle, it will be reflected and the scattered acoustic
wave will be generated in the open space. Meanwhile, an elastic wave is induced
inside the obstacle. This scattering phenomenon leads to a coupled
acoustic-elastic interaction problem. The surface divides the whole space into
two parts: the interior of the obstacle and the exterior of the obstacle. The
wave propagation is governed by the elastic wave equation and the acoustic wave
equation in these two regions, respectively. The acoustic and elastic wave
equations are coupled on the surface through two continuity conditions: the
kinematic interface condition and the dynamic condition. The dynamic interaction
between an elastic structure and surrounding acoustic medium is
encountered in many areas of engineering and industrial design and
identification \cite{FG-07, FG-S-98, MO-95, OS-97}, such as detection of
submerged objects, vibration analysis for aircrafts and automobiles, and
ultrasound vibro-acoustography.

The acoustic-elastic interaction problems have continuously attracted much
attention by many researchers. There are a lot of available mathematical and
numerical results, especially for the time-harmonic wave equations \cite{D-89,
DGH-CMAME-82, F-JASA-51, HJ-86, H-94, HKR-MN-00, LM-SIAP-95, HKY-IPI-16,
YHXZ-IP-16, SM-JCP-06, SG-CMAME-06}. The time-domain problems have received
considerable attention due to their capability of capturing wide-band signals
and modeling more general material and nonlinearity \cite{CM-SIMA-14,
RJ-IEEE-08, WWZ-SIAP-12}. Many approaches are attempted to solve numerically the
time-domain problems such as coupling of boundary element and finite element
with different time quadratures \cite{EA-NME-91, FKW-NME-06, SM-JCP-06,
LM-JCP-15, HQSS-IEA-17, MMS-JCP-04}. Compared with the time-harmonic scattering
problems, the time-domain problems are less studied on their rigorous
mathematical analysis due to the additional challenge of the temporal
dependence. The analysis can be found in \cite{CN-JCM-08, LWW-SIAP-15,
GL-JDE-16, GL-M3AS-17} for the time-domain acoustic and electromagnetic
scattering problems in different structures including bounded obstacles,
periodic surfaces, and unbounded rough surfaces. We refer to \cite{GLL-M2AS-18}
and \cite{GLZ-SIMA-17} for the mathematical analysis of the transient
elastic scattering problems in an unbounded structure.

The wave scattering usually involves exterior boundary value problems such as
the acoustic-elastic interaction problem discussed in this paper. The unbounded
domains need to be truncated into bounded ones regardless of
mathematical analysis or numerical computation. Therefore, appropriate boundary
conditions are required on the boundaries of the truncated domains so that no
artificial wave reflection occur. Such boundary conditions are called
transparent boundary conditions (TBCs) or non-reflecting boundary conditions
(NRBCs) \cite{N-01}. They are the subject matter of much ongoing research
\cite{AGH-JCP-02, G-92, GK-SIAP-95, H-AN-99}. The research on the perfectly
matched layer (PML) technique has undergone a tremendous development since
Berenger proposed a PML for solving the Maxwell equations \cite{B-JCP-94,
TY-ANM-98}. The basic idea of the PML technique is to surround the domain of
interest by a layer of finite thickness fictitious material which absorbs all
the waves coming from inside the computational domain. When the waves reach the
outer boundary of the PML region, their values are so small that the homogeneous
Dirichlet boundary conditions can be imposed. Comparing with the PML method for
the time-harmonic scattering problems, the rigorous mathematical analysis is
much more sophisticated for the time-domain PML method due to challenge of the
dependence of the absorbing medium on all frequencies \cite{C-NAM-09,
CW-SINUM-12, HBR-IEEE-02, DJ-CMAME-06, KK-CMAME-11}.

Recently, Bao et al. has done some mathematical analysis for the time-domain
acoustic-elastic interaction problem in two dimensions \cite{BGL-ARMA-18}. The
problem was reformulated into an initial-boundary value problem in a bounded
domain by employing a time-domain TBC. Using the Laplace transform and energy
method, they showed that the reduced variational problem has a unique weak
solution in the frequency domain and obtain the stability estimate for the
solution in the time-domain. An a priori estimates with explicit time dependence
was achieved for the solution of the time-domain variational problem. In
addition, the PML method was discussed and a first order symmetric hyperbolic
system was considered for the truncated PML problem. It was shown that the
system has a unique strong solution and the stability is obtained for the
solution. However, the convergence analysis is lacking for the time-domain PML
problem at present.

In this paper, we carry the mathematical analysis for the two- and
three-dimensional acoustic-elastic interaction problem by using a different
method. It is known that waves have finite speed of propagation in the
time-domain, which differs from the infinite speed of propagation for
time-harmonic waves. We make use of this fact and propose a compressed
coordinate transformation to reduce the problem equivalently into an
initial-boundary value problem in a bounded domain. Given any time $T$, we
consider the problem in the time interval $(0, T]$. The method begins with
constructing an annulus to surround the obstacle. The inner sphere can be
chosen as close as possible to the obstacle, but the radius of the outer
sphere should be chosen sufficiently large so that the scattered acoustic 
wave cannot reach it at time $t=T$. Hence the homogeneous Dirichlet boundary
condition can be imposed on the outer sphere. But the domain may be too large
for actual computation. To overcome this issue, we apply the change of
variables and compress the annulus into a much smaller annulus by mapping the
outer sphere into a sphere which is slightly larger than the inner sphere while
keeping the inner sphere unchanged. The reduced problem can be formulated in a
much smaller domain where the obstacle is only enclosed by a thin annulus. Based
on the Galerkin method and energy estimates, we prove the existence and
uniqueness of the weak solution for the corresponding variational problem.
Furthermore, we obtain a priori estimate with explicit dependence on the time
for the pressure of the acoustic wave and the displacement of the elastic wave.
The method does not introduce any approximation or truncation error. It avoids
the complicated error or convergence analysis which needs to be carefully done
for the TBC or PML method. Therefore, the reduced model problem is also
particularly suitable for numerical simulations due to its simplicity and small
computational domain.

The paper is organized as follows. In Section 2, we introduce the model
equations for the acoustic-elastic interaction problem on and propose the
compressed coordinate transformation to reduce the problem into an
initial-boundary value problem. Section 3 is devoted to the analysis of the
reduced problem, where the well-posedness and stability are addressed, and an a
priori estimates with explicit time dependence is obtained for the time-domain
variational problem. The paper is concluded with some general remarks in Section
4. To avoid distraction from the presentation of the main results, we give in
Appendices the details of the change of variables for the compressed
coordinate transformation.

\section{Problem formulation}

In this section, we introduce the problem geometry and model equations, and
propose a compressed coordinate transformation to reduce the acoustic-elastic
scattering problem into an initial boundary value problem in a bounded domain
over a finite time interval.

\subsection{Problem geometry}

Consider a bounded elastic obstacle which may be described by the bounded domain
$D\subset\mathbb R^d$ with Lipschitz continuous boundary $\partial D$, where
$d=2$ or $3$. We assume that $D$ is occupied by an isotropic linearly elastic
medium which is characterized by a constant mass density $\rho_{2}>0$ and
Lam\'{e} parameters $\lambda, \mu$ satisfying $\mu>0, \lambda+\mu>0$.
The obstacle's surface divides the whole space $\mathbb R^d$ into the interior
domain $D$ and the exterior domain $\mathbb R^d\setminus\bar D$. The
elastic wave and the acoustic wave propagates inside $D$ and $\mathbb
R^d\setminus\bar D$, respectively. The exterior domain $\mathbb
R^d\setminus\bar{D}$ is assumed to be connected and filled with a homogeneous,
compressible, and inviscid air or fluid with a constant density $\rho_{1}>0$. It
is known that the acoustic wave has a finite speed of propagation in the time
domain. Hence, for any given time $T>0$, we may always pick a sufficiently large
$R>0$ such that the acoustic wave cannot reach the surface
$\partial B_R=\{x\in\mathbb R^d: |x|=R\}$. Denote the
ball $B_a=\{x\in\mathbb R^d: |x|<a\}$ with
the boundary $\partial B_{a}=\{x\in\mathbb R^d:
|x|<a\}$, where $a>0$ is a constant such that $\bar{D}\subset B_a$.
Usually we have $a\ll R$. Let $b$ be an appropriate constant satisfying $a<b\ll
R$. Define $B_b=\{x\in\mathbb R^d: |x|<b\}$ and
$\partial B_b=\{x\in\mathbb R^d: |x|=b\}$. We shall
consider a compressed coordinate transformation which compresses the annulus
$\{x\in\mathbb R^d: a<|x|<R\}$ into the much smaller
annulus $\{x\in\mathbb R^d: a<|x|<b\}$ by mapping
$\partial B_R$ into $\partial B_b$ while keeping $\partial B_a$ unchanged. Then
the acoustic-elastic interaction problem will be formulated in the bounded
domain $B_b$. The problem geometry is shown in Figure 1.

\begin{figure}\label{geometry}
\centering
\includegraphics[width=0.35\textwidth]{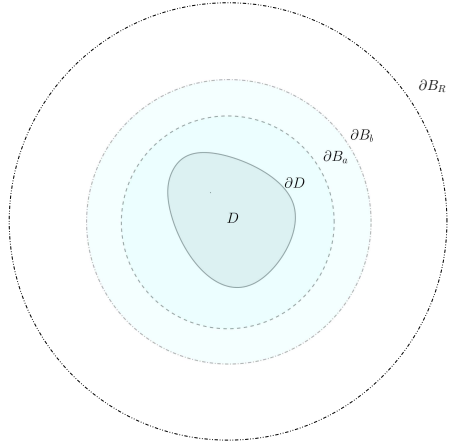}
\caption{Problem geometry of the acoustic scattering by a bounded elastic
obstacle.}
\end{figure}

\subsection{The model equations}

Let the obstacle be illuminated by an acoustic plane wave $p^{\rm inc}(x,
t)=\vartheta(ct-\theta\cdot x)$, where $\vartheta$ is a smooth function with a
compact support and $\theta\in\mathbb S^{d-1}$ is a unit propagation direction
vector. The acoustic wave field in $\mathbb R^d\setminus\bar{D}$ is governed by
the conservation and the dynamics equations in the time domain:
\begin{equation}\label{cde}
 \nabla p(x, t)=-\rho_{1}\partial_t \boldsymbol{v}(x, t),
 \quad c^{2}\rho_{1} \nabla\cdot\boldsymbol{v}(x, t)=-\partial_t p(x, t),\quad
x\in\mathbb R^d\setminus\bar{D},~ t>0,
\end{equation}
where $p$ is the pressure, $\boldsymbol{v}$ is the velocity, $\rho_{1}>0$ and
$c>0$ are the density and wave speed, respectively. Eliminating the velocity
$\boldsymbol{v}$ from \eqref{cde}, we may easily verify that the pressure $p$
satisfies the acoustic wave equation
\begin{equation}\label{aw}
 \frac{1}{c^2}\partial_t^2 p(x, t)-\Delta p(x, t)=0,\quad x\in\mathbb
R^d\setminus\bar{D},~ t>0.
\end{equation}

The scattered field $p^{\rm sc}=p-p^{\rm inc}$ is excited due to the interaction
between the incident field and the obstacle. It follows from \eqref{aw} and the
expression of the plane incident wave $p^{\rm inc}$ that the scattered field
$p^{\rm sc}$ also satisfies the acoustic wave equation
\begin{equation}\label{aws}
 \frac{1}{c^2}\partial_t^2 p^{\rm sc}(x, t)-\Delta p^{\rm sc}(x, t)=0,\quad
x\in\mathbb R^d\setminus\bar{D}, ~ t>0.
\end{equation}
By assuming that the incident field vanishes for $t\leq 0$, i.e., the system is
assumed to be quiescent at the beginning, we may impose the homogeneous initial
conditions for the scattered field
\[
p^{\rm sc}|_{t=0}=\partial_t p^{\rm sc}|_{t=0}=0\quad\text{in}~\mathbb
R^d\setminus\bar{D}.
\]
Since the acoustic wave in \eqref{aws} has a finite speed of propagation, for
any given time $T>0$, we may always pick a sufficiently large $R>0$ such
that the scattered field $p^{\rm sc}$ cannot reach the surface $\partial B_R$,
i.e., the homogeneous Dirichlet boundary condition can be imposed
\[
p^{\rm sc}=0\quad\text{on } \partial B_R\times (0, T].
\]

Recall that the domain $D$ is occupied by a linear and isotropic elastic body.
Under the hypothesis of small amplitude oscillations in the obstacle, the
elastic wave satisfies the linear elasticity equation
\begin{equation}\label{lee}
 \nabla\cdot\boldsymbol{\sigma}(\boldsymbol{u}(x,t))
-\rho_{2}\partial_t^{2}(\boldsymbol{u}(x,t)) = 0,
 \quad x\in D,~ t>0.
\end{equation}
where $\boldsymbol{u}=(u_1, \dots, u_d)^\top$ is the displacement vector,
$\rho_{2}>0$ is the density, and the Cauchy stress tensor $\boldsymbol{\sigma}$
is given by the generalized Hook law:
\begin{equation}\label{hook}
\boldsymbol{\sigma}(\boldsymbol{u}) = 2\mu\boldsymbol{\epsilon(u)}+\lambda {\rm
tr}(\boldsymbol{\epsilon(u)}) I, \quad
\boldsymbol{\epsilon(u)}=\frac{1}{2}\big(\nabla\boldsymbol{u}
+(\nabla\boldsymbol{u})^{\top}\big).
\end{equation}
Here the Lam\'{e} constants $\mu, \lambda$ satisfy $\mu>0, \lambda+\mu >0$,
$I$ is the identity matrix, $\boldsymbol{\epsilon(u)}$ is known as the
strain tensor, and $\nabla\boldsymbol{u}$ is the displacement gradient tensor
defined by
\[
\nabla\boldsymbol{u}=\begin{bmatrix}
    \partial_{x_{1}}u_{1} & \cdots&  \partial_{x_{d}}u_{1}\\
    \vdots                & \ddots&  \vdots\\
    \partial_{x_{1}}u_{d} & \cdots&  \partial_{x_{d}}u_{d}
   \end{bmatrix}.
\]
Substituting \eqref{hook} into \eqref{lee}, we obtain the time-domain Navier
equation
\[
\mu \Delta \boldsymbol{u}(x, t)+(\lambda+\mu)\nabla\nabla\cdot\boldsymbol{u}(x,
t) -\rho_{2}\partial_t^{2}\boldsymbol{u}(x,t) = 0, \quad x\in D,~ t>0.
\]
Since the system is assumed to be quiescent, the displacement vector is
constrained by the homogeneous initial conditions:
\[
\boldsymbol{u}(x,t)|_{t=0}=\partial_t \boldsymbol{u}(x,t)|_{t=0}=0, \quad x\in
D.
\]

To describe the coupling of acoustic and elastic waves at the interface,
the kinematic interface condition is imposed to ensure the continuity of the
normal component of the velocity on $\partial D$:
\[
\boldsymbol{n}_{D}\cdot\boldsymbol{v}(x,t)=\boldsymbol{n}_{D}\cdot\partial_t
\boldsymbol{u}(x,t), \quad x\in \partial D, ~ t>0,
\]
where $\boldsymbol n_D$ is the unit normal vector on $\partial D$ pointing
towards $\mathbb R^2\setminus \bar D$. Noting
$-\rho_{1}\partial_{t}\boldsymbol{v}(x,t)=\nabla p(x,t)$, we have
\[
\partial_{\boldsymbol{n}_{D}} p(x,t)=\boldsymbol{n}_{D} \cdot \nabla
p(x,t)=-\rho_{1}\boldsymbol{n}_{D}\cdot\partial_t^{2} \boldsymbol{u}(x,t),
\quad x\in \partial D, ~ t>0.
\]
In addition, the following dynamic interface condition is required
\[
 -p(x,t)\boldsymbol{n}_{D}=\mu \partial_{\boldsymbol{n}_{D}} \boldsymbol{u}(x,
t)
 +(\lambda+\mu)(\nabla\cdot\boldsymbol{u}(x, t))\boldsymbol{n}_{D}, \quad x\in
\partial D, ~ t>0.
\]

To summarize, the acoustic scattering by an elastic obstacle can be formulated
as an initial boundary value problem in the bounded domain $B_R$ over the finite
time interval $(0, T]$:
\begin{equation}\label{ibvpo}
\begin{cases}
 \frac{1}{c^2}\partial_t^2 p-\Delta p=0,&\quad \text{in}~
B_R\setminus\bar{D}\times (0, T],\\
p=p^{\rm inc},&\quad\text{on}~ \partial B_R\times (0, T],\\
p|_{t=0}=\partial_t p|_{t=0}=0,&\quad\text{in}~B_R\setminus\bar{D},\\
\mu\Delta\boldsymbol{u}+(\lambda+\mu)\nabla\nabla\cdot\boldsymbol{u}
-\rho_{2}\partial_t^{2}\boldsymbol{u}=0, &\quad \text{in}~ D\times (0, T],\\
\boldsymbol{u}|_{t=0}=\partial_t \boldsymbol{u}|_{t=0}=0, &\quad \text{in}~ D,\\
\partial_{\boldsymbol{n}_{D}} p=-\rho_{1}\boldsymbol{n}_{D}\cdot\partial_t^{2}
\boldsymbol{u}, \quad -p\boldsymbol{n}_{D}=\mu \partial_{\boldsymbol{n}_{D}}
\boldsymbol{u}+(\lambda+\mu)(\nabla\cdot\boldsymbol{u})\boldsymbol{n}_{D},
&\quad \text{in}~ \partial D\times (0, T].
\end{cases}
\end{equation}

Now we introduce some useful notation. The scalar,
vector, and matrix real-valued $L^2$ inner products are defined by
\[
(a,b)_{D}:=\int_{D}a b\,{\rm d}x, \quad
(\boldsymbol{a},\boldsymbol{b})_{D}:=\int_{D}
\boldsymbol{a} \cdot \boldsymbol{b}\,{\rm d}x,\quad
(\boldsymbol{A},\boldsymbol{B})_{D}:=\int_{D} \boldsymbol{A} :
\boldsymbol{B}\,{\rm d}x,
\]
where the colon denotes the Frobenius inner product of square matrices, i.e.,
$\boldsymbol{A} : \boldsymbol{B}= {\rm tr} (AB^{\top})$. When using
complex-valued functions, the complex conjugate will be used as needed. Let
$\Omega$ be a bounded open domain with Lipschitz boundary $\partial\Omega$.
Denote by $L^2(\Omega)$ the space of square integrable functions in $\Omega$
equipped with the norm $\|\cdot\|_{L^2(\Omega)}$. Let $H^s(\Omega), s\in\mathbb
R$ be the standard Sobolev space equipped with the norm
$\|\cdot\|_{H^s(\Omega)}$. Denote $\boldsymbol{L}^{2}(D)={L}^{2}(D)^{d}$,
$\boldsymbol{H}^{1}(D)={H}^{1}(D)^{d}$, and ${\boldsymbol{L}^{2} (\partial
D)}={L}^{2} (\partial D)^{d}$, which have norms characterized by
\[
\|\boldsymbol{u}\|^2_{\boldsymbol{L}^{2}(D)}=\sum_{j=1}^{d}\|u_{j}\|^{2}_{L^{2}
(D)}, \quad
\|\boldsymbol{u}\|^2_{\boldsymbol{H}^{1}(D)}=\sum_{j=1}^{d}\|u_{j}\|^{2}_{H^{1}
(D)}, \quad
\|\boldsymbol{u}\|^2_{\boldsymbol{L}^{2} (\partial
D)}=\sum_{j=1}^{d}\|u_{j}\|^{2}_{{L}^{2} (\partial D)}.
\]
The 2-norm of the gradient tensor is defined by
\[
\|\nabla\boldsymbol{u}\|^2_{L^{2}(D)^{d\times d}}=\sum_{j=1}^{d}\int_{D}|\nabla
u_{j}|^{2}{\rm d}x.
\]

\subsection{The reduced problem}

In this section, we propose the compressed coordinate transformation to reduce
equivalently the acoustic-elastic interaction problem \eqref{ibvpo} into an
initial boundary value problem in a much smaller domain $B_b$. Although $R$ is
chosen to be large enough so that the scattered wave cannot reach $\partial
B_R$, $b$ does not have to be large as long as $b>a$. The width of the annulus
$b-a$ can be small and the annulus $B_b\setminus \bar{B}_a$ can be put as close
as possible to enclose the obstacle $D$, which makes it particularly attractive
for the numerical simulation.

Consider the change of variables
\[
\rho=\zeta(r)=\begin{cases}
           r, &\quad r\in [0, a),\\
           \eta(r),&\quad r\in [a, b],
          \end{cases}
\]
where
\[
 \eta(r)=\frac{\xi(r)}{(b-r)(R-b)+(b-a)^2},\quad \xi(r)=a^2(R-b)+r(a^2+(b-2a)R).
\]
A simple calculation yields
\[
 \eta'(r)=\frac{(R-a)^2(b-a)^2}{((b-r)(R-b)+(b-a)^2)^2}.
\]
It is clear to note that
\[
 \eta(a)=a,\quad \eta(b)=R, \quad \eta'(a)=1,
\]
which imply that the function $\zeta\in C^1[0, b]$ is positive and monotonically
increasing, i.e., $\zeta>0$ and $\zeta'>0$. Hence, the transform $\zeta$ keeps
the ball $B_a$ to itself while compresses the annulus $B_R\setminus\bar B_a$
into the annulus $B_b\setminus\bar B_a$. Define $\Omega=B_b\setminus\bar D$ and
its boundary $\partial\Omega=\partial D\cup\partial B_b$.

Let $v$ be the transformed scattered field of $p$ under the change of
variables. It follows from the appendices that $v$ satisfies
\[
\frac{\beta}{c^2}\partial_t^2 v-\nabla\cdot(M\nabla
v)=0\quad\text{in}~\Omega\times(0, T],
\]
where the variable coefficients
\[
 \beta=\frac{\zeta\zeta'}{r},\quad M=Q\begin{bmatrix}
     \frac{\zeta}{r\zeta'} & 0\\
     0 & \frac{r\zeta'}{\zeta}
    \end{bmatrix}Q^\top,\quad
 Q=\begin{bmatrix}
    \cos\theta & -\sin\theta\\
    \sin\theta & \cos\theta
   \end{bmatrix}\quad\text{for } d=2
\]
and
\[
 \beta=\frac{\zeta^2}{r^2}, \quad M=Q\begin{bmatrix}
\frac{\zeta^2}{r^2\zeta'} & 0 & 0 \\
 0 & \zeta' & 0\\
0 & 0 & \zeta'
\end{bmatrix}Q^\top,\quad  Q=\begin{bmatrix}
   \sin\theta\cos\varphi & \cos\theta\cos\varphi & -\sin\varphi\\
\sin\theta\sin\varphi & \cos\theta\sin\varphi & \cos\varphi\\
\cos\theta & -\sin\theta & 0
   \end{bmatrix}\quad\text{for } d=3.
\]
Here $(r, \theta)$ and $(r, \theta, \varphi)$ are the
polar and spherical coordinates in the two- and three-dimensions, respectively.
It is easy to note that $\beta$ is a continuous positive function, $Q$ is an
orthonormal matrix, and $M$ is a symmetric positive definite matrix with
continuous matrix entries. The details are given in the appendices.

For the given $p^{\rm inc}$, there exists a smooth lifting $v_0$ which has a
compact support contained in $\Omega\times[0, T]$ and satisfies the boundary
conditions $v_0=p^{\rm inc}$ on $\partial B_b$. Hence we may equivalently
consider the following initial boundary value problem
\begin{subequations} \label{ibvp}
\begin{numcases}
{ }
\frac{\beta}{c^2}\partial_t^2 p-\nabla\cdot(M\nabla
p)=f \label{eq1} &\text{in } $\Omega\times(0, T]$,  \\
p=0 \label{bv1} &\text{on } $\partial B_b\times (0, T]$,\\
p|_{t=0}=g,\quad\partial_t p|_{t=0}=h \label{iv1} &\text{in } $\Omega$,\\
\mu\Delta\boldsymbol{u}+(\lambda+\mu)\nabla\nabla\cdot\boldsymbol{u}
-\rho_{2}\partial_t^{2}\boldsymbol{u}=0 \label{eq2} &\text{in } $D\times (0,
T]$,\\
\boldsymbol{u}|_{t=0}=\partial_t \boldsymbol{u}|_{t=0}=0  \label{iv2}
&\text{in } $D$,\\
\partial_{\boldsymbol{n}_{D}} p=-\rho_{1}\boldsymbol{n}_{D}\cdot\partial_t^{2}
\boldsymbol{u}, \quad -p\boldsymbol{n}_{D}=\mu \partial_{\boldsymbol{n}_{D}}
\boldsymbol{u}+(\lambda+\mu)(\nabla\cdot\boldsymbol{u})\boldsymbol{n}_{D}
\label{bv2}&\text{in }  $\partial D\times (0, T]$.
\end{numcases}
\end{subequations}
where $f\in L^2(\Omega), g\in \widetilde{H}_{0}^1(\Omega):=\{u\in H^1(\Omega):
u=0~\text{on}~\partial B_{b}\}, h\in L^2(\Omega)$.

\section{Well-posedness}

In this section, we examine the well-posedness of the reduced initial-boundary
value problem \eqref{ibvp} and present an a priori estimate for the solution.

\subsection{Existence and uniqueness}

Taking the inner products in \eqref{eq1} and \eqref{eq2} with the
test functions $q \in \widetilde{H}_{0}^{1}(\Omega)$ and $\boldsymbol{v} \in
\boldsymbol{H}^{1}(D)$, respectively, we arrive at the variational problem:
to find $(p, \boldsymbol{u}) \in \widetilde{H}^{1}_{0}(\Omega)\times
\boldsymbol{H}^{1}(D)$ for all $t > 0$ such that
\[
\frac{\beta}{c^{2}}\big(\partial_t^2{p},
q\big)_{\Omega}-\big(\nabla\cdot(M\nabla
p), q \big)_{\Omega}=(f, q)_{\Omega},\quad\forall\, q\in
\widetilde{H}_{0}^{1}(\Omega),
\]
\[
\rho_{2}\big(\partial_t^2\boldsymbol{u},
\boldsymbol{v}\big)_{D}-\big(\mu\Delta\boldsymbol{u}
+(\lambda+\mu)\nabla\nabla\cdot\boldsymbol{u}, \boldsymbol{v}
\big)_{D}=0,\quad\forall\, \boldsymbol{v} \in
\boldsymbol{H}^{1}(D).
\]
Using the integration by parts and initial and boundary conditions \eqref{iv1},
\eqref{iv2}, \eqref{bv1}, and \eqref{bv2}, we have
\[
\big(\frac{\beta}{c^{2}}\partial_t^2{p},
q\big)_{\Omega}+a_{0}[p,q;t]-\int_{\partial D} \rho_{1}(\boldsymbol{n}_{D}\cdot
\partial_t^2\boldsymbol{u}) q{\rm d}s=(f, q)_{\Omega},\quad \forall\, q\in
\widetilde{H}_{0}^{1}(\Omega),
\]
\[
\big(\rho_{2}\partial_t^2{\boldsymbol{u}},
\boldsymbol{v}\big)_{D}+a_{1}[\boldsymbol{u}, \boldsymbol{v};t]+\int_{\partial
D} p\boldsymbol{n}_{D}\cdot \boldsymbol{v} {\rm d}s=0,\quad \forall\,
\boldsymbol{v} \in \boldsymbol{H}^{1}(D),
\]
where the bilinear forms
\[a_{0}[p,q;t]=\int_\Omega (M^{1/2}\nabla p)\cdot(M^{1/2}\nabla q) {\rm d} x,
 \]
\[
 a_{1}[\boldsymbol{u}, \boldsymbol{v};t]=\mu\int_{D} (\nabla
\boldsymbol{u}):(\nabla \boldsymbol{v}) {\rm d} x
 +(\lambda+\mu)\int_{D} (\nabla \cdot \boldsymbol{u})(\nabla \cdot
\boldsymbol{v}) {\rm d} x.
\]

Suppose that $(p(x, t), \boldsymbol{u}(x, t))$ is a smooth solution of
\eqref{ibvp} and define the associated mappings ${\rm p}:[0, T]\to
\widetilde{H}_0^1(\Omega)$ and ${\bf u}:[0, T]\to \boldsymbol{H}^1(D)$
by
\begin{align*}
[{\rm p}(t)](x):&=p(x, t),
\quad x\in\Omega, ~t\in[0, T],\\
[{\mathbf u}(t)](x):&=\boldsymbol{u}(x, t),
\quad x\in D, ~t\in[0, T].
\end{align*}
Introduce the function ${\rm f}:[0, T]\to L^2(\Omega)$ by
\[
 [{\rm f}(t)](x):= f(x, t),\quad x\in\Omega, ~ t\in[0, T].
\]

We seek a weak solution $({\rm p}, {\bf u})$ satisfying $({\rm p}'', {\bf u}'')
\in H^{-1}(\Omega)\times\boldsymbol{H}^{-1}(D)$ for a.e. $t\in[0,
T]$. Hence the inner product $(\cdot,\cdot)$ can also be interpreted as the
pairing $\langle\cdot,\cdot\rangle$ which is defined between the dual spaces of
$H^{-1}$ and $H^1$.

\begin{definition}
 We say that the function $({\rm p}, {\bf u})\in L^2(0, T;
\widetilde{H}^1_0(\Omega))\times L^2(0, T; \boldsymbol{H}^{1}(D))$ with $({\rm
p}', {\bf u}')\in L^2(0, T; L^2(\Omega))\times L^2(0, T; \boldsymbol{L}^{2}(D))$
and $({\rm p}'', {\bf u}'')\in L^2(0, T; H^{-1}(\Omega))\times L^2(0, T;
\boldsymbol{H}^{-1}(D))$
is a weak solution of the initial boundary value problem \eqref{ibvp} if it
satisfies

\begin{enumerate}[label=(\roman*)]

\item $\forall {\rm q}\in \widetilde{H}_0^1(\Omega)$, $ {\bf{v}} \in
{\boldsymbol{H}}^{1}{(D)}$ and a.e. $t\in[0, T]$,
\begin{align*}
&\big(\frac{\beta}{c^{2}}{\rm p} '', \rm q\big)_{\Omega}
+\big(\rho_{1}\rho_{2}{\bf{u}} '', {\bf{v}}\big)_{D}
+a_{0}[{\rm p}, {\rm q};t]\\
&\hspace{2cm}+\rho_{1}\big(a_{1}[{\bf{u}}, {\bf{v}};t]+a_{2}[{\rm p},
{\bf{v}};t]+a_{3}[\bf{u}, \rm q;t]\big)=(\rm{ f, q})_{\Omega},
\end{align*}
where
\begin{align*}
a_{2}[{\rm p}, {\bf v}; t]&=\int_{\partial D}
{\rm p}\boldsymbol{n}_{D}\cdot {\bf v}{\rm d} s,\\
a_{3}[{\bf u}, {\rm q}; t]&=\int_{\partial D} -(\boldsymbol{n}_{D}\cdot {\bf
u}'') {\rm q} {\rm d}s.
\end{align*}

\item ${\rm p}(0)=g,~ {\rm p}'(0)=h$.

\end{enumerate}
\end{definition}

We adopt the Galerkin method to construct the weak solution of the
initial boundary value problem \eqref{ibvp} by solving a finite dimensional
approximation. We refer to \cite{E-10} for the method to construct the weak
solutions of the general second order parabolic and hyperbolic equations. The
method begins with selecting orthogonal basis
functions: select functions $w_k:=(w_k^{i}(x),w_k^{e}(x))^{\top}, k\in\mathbb
N$ by requiring that the smooth functions $\{w_k^{i}\}_{k=1}^\infty$,
$\{w_k^{e}\}_{k=1}^\infty$ is the standard
orthogonal basis of $L^2(\Omega)$ and $L^{{2}}(\partial\Omega)$ respectively,
and $\{w_k^{i}\}_{k=1}^\infty$ is also the orthogonal basis of
$H_{0}^{1}(\Omega)$; select functions $W_k:=(W_k^{i}(x),W_k^{e}(x))^{\top},
k\in\mathbb N$ by requiring that the smooth functions
$\{W_k^{i}\}_{k=1}^\infty$, $\{W_k^{e}\}_{k=1}^\infty$ is the standard
orthogonal basis of $\boldsymbol{L}^{2}(D)$ and $\boldsymbol{L}^{{2}}(\partial
D)$ respectively, and $\{W_k^{i}\}_{k=1}^\infty$ is also the orthogonal basis of
$\boldsymbol{H}_{0}^{1}(D)$.

For positive integers $s, l$, $N:=s+l$, $w_k^{e}=0\, (k=1, \dots, s)$,
$w_k^{i}=0\, (k=s+1, \dots, N)$, let
\begin{equation}\label{pmt}
 {\rm p}_N(t):=\sum_{j=1}^N p_{Nj}(t) w_j=\sum_{j=1}^s p_{Nj}^{i}(t)
w_j^{i}+\sum_{j=s+1}^N p_{Nj}^{e}(t) w_j^{e}.
\end{equation}
For positive integers $m, n$, $M:=m+n$, $W_k^{i}=0\,(k=1, \cdots, m)$,
$W_k^{e}=0\, (k=m+1, \cdots, M)$, let
\begin{equation}\label{umt}
 {\bf u}_M(t):=\sum_{j=1}^M u_{Mj}(t) W_j=\sum_{j=1}^m u_{Mj}^{i}(t)
W_j^{i}+\sum_{j=m+1}^M u_{Mj}^{e}(t) W_j^{e}.
\end{equation}
The coefficients $ p_{Nj}(t), u_{Mj}(t)$ satisfy the initial conditions
\begin{align}\label{puic}
p_{Nj}(0)=(g,  w_j),\quad p_{Nj}^{'}(0)=(h, w_j),\quad
u_{Mj}(0)=0, \quad  u_{Mj}^{'}(0)=0,
\end{align}
and ${\rm p}_N(t), {\bf u}_M(t)$ satisfy the equation
\begin{align}\label{vf6}
({\rm{f}}, w_k)_{\Omega}=
\big(\frac{\beta}{c^{2}}{\rm p}_N^{''}, w_k\big)_{\Omega}
&+\big(\rho_{1}\rho_{2}{\bf{u}}_{M} '', W_{j}\big)_{D}
+a_{0}[{\rm p}_N, w_k;t]\nonumber\\
&+\rho_{1}\big(a_{1}[{\bf{u}}_{M},  W_{j};t]+a_{2}[{\rm p}_N,
W_{j};t]+a_{3}[{\bf{u}}_{M}, w_k;t]\big),
\end{align}
for $k=1, \dots, N, j=1, \dots, M, t\in[0, T]$.

\begin{theorem}
 For each $M, N \in\mathbb N$, there exists unique functions ${\rm p}_N, {\bf
u}_M$, which are given in the form of \eqref{pmt}--\eqref{umt} and satisfy
\eqref{puic}--\eqref{vf6}.
\end{theorem}

\begin{proof}
 Since $\{w_k^{i}\}_{k=1}^\infty$ and $\{w_k^{e}\}_{k=1}^\infty$ are the
orthogonal bases of $L^2(\Omega)$ and $L^{{2}}(\partial\Omega)$;
$\{W_k^{i}\}_{k=1}^\infty$ and $\{W_k^{e}\}_{k=1}^\infty$ are the orthogonal
bases of $\boldsymbol{L}^{2}(D)$ and $\boldsymbol{L}^{{2}}(\partial D)$, we have
from \eqref{pmt}--\eqref{umt} that
\begin{align}
\label{pN2} ({\rm p}_N^{''}(t), w_k)_{\Omega}&=p_{Nk}^{''}(t), \quad k=1, \dots,
N,\\
\label{uN2} ({\bf u}_M^{''}(t), W_j)_{D}&=u_{Mj}^{''}(t), \quad j=1, \dots, M.
\end{align}
It follows from \eqref{vf6} that
\begin{align}
\label{a0} a_{0}[{\rm p}_N, w_k;t]&=\sum_{j=1}^{N}d_{k}^{j}(t)p_{Nj}(t),\\
\label{a1} a_{1}[{\bf u}_M, W_k;t]&=\sum_{j=1}^{M}c_{k}^{j}(t)u_{Mj}(t),
\end{align}
where $d_{k}^{j}(t)=a_{0}[w_j, w_k; t]$, $j, k=1, \dots, N$ and
$c_{k}^{j}(t)=a_{1}[W_j, W_k; t]$, $j, k=1, \dots, M$. Define
$D=[d_{k}^{j}]_{N\times N}$ and $C=[c_{k}^{j}]_{M\times M}$.

Recall that the matrix $M$ is symmetric positive definite. It follows from
the definition of $a_{0}[p,q;t]$ and $a_{1}[\boldsymbol{u}, \boldsymbol{v};t]$
that there exists  positive constants $C_j, j=1, \dots, 4$ such that
\[
 C_1\|p\|^2_{H^1(\Omega)}\leq | a_{0}[p,p;t]| \leq
C_2\|p\|^2_{H^1(\Omega)},\quad\forall\, p\in H_0^1(\Omega)
\]
and
\[
 C_3\|\boldsymbol{u}\|^2_{\boldsymbol{H}^1(D)}\leq | a_{1}[\boldsymbol{u},
\boldsymbol{u};t]| \leq
C_4\|\boldsymbol{u}\|^2_{\boldsymbol{H}^1(D)},\quad\forall\, \boldsymbol{u} \in
{\boldsymbol{H}_{0}^1(D)},
\]
which imply that the bilinear forms $a_{0}$ and $a_{1}$ are coercive in
$H_0^1(\Omega)$ and ${\boldsymbol{H}_{0}^1(D)}$, respectively, i.e., there
exists a positive constant $C$ such that
\begin{align*}
 a_{0}[p, p; t] & \geq C\|p\|^2_{H^1(\Omega)},\quad\forall\, p\in
H_0^1(\Omega),\\
 a_{1}[\boldsymbol{u}, \boldsymbol{u};t] & \geq
C\|\boldsymbol{u}\|^2_{\boldsymbol{H}^1(D)},\quad\forall\, \boldsymbol{u}\in
\boldsymbol{H}_{0}^1(D).
\end{align*}

Similarly, we have from \eqref{vf6} that
\begin{align}
\label{a2} a_{2}[{\rm p}_N,
W_k;t] &=\rho_{1}\sum_{j=1}^{N}e_{k}^{j}(t)p_{Nj}(t),\\
\label{a3} a_{3}[{\bf u}_M,
w_k;t] &=\rho_{1}\sum_{j=1}^{M}l_{k}^{j}(t)u_{Mj}^{''},
\end{align}
where $e_{k}^{j}(t)=(w_j\boldsymbol{n}_{D}, W_k)_{\partial D}$, $j=1, \dots, N,
k=1, \dots, M$ and $l_{k}^{j}(t)=(\boldsymbol{n}_{D}\cdot W_j,w_k)_{\partial
D}$, $j=1, \dots, M, k=1, \dots, N$. Define $E=[e_{k}^{j}]_{n\times
l}$ and $L=[l_{k}^{j}]_{l\times n}$.

Let
\[
 f^k(t)=( {\rm f}(t), w_k)_{\Omega},\quad k=1, \dots, N.
\]
Substituting \eqref{pN2}--\eqref{a3} into \eqref{vf6}, we
obtain a linear system of second order equations
\begin{equation}\label{ode}
 A{\bf U}''_{\overline{M}}+B{\bf U}_{\overline{M}}=F
\end{equation}
subject to the initial conditions \eqref{puic}, where ${\overline{M}}=M+N$,
\[
A=\begin{array}{c@{\hspace{-5pt}}l}
\left(\begin{array}{cc;{2pt/2pt}c;{2pt/2pt}cc;{2pt/2pt}c}
\rho_{1}\rho_{2}I_{m\times m}
&  &  &  &  &  \\ \hdashline[2pt/2pt]
&  &  &  &  \\
& & \rho_{1}\rho_{2}I_{n\times n}& & & \\ \hdashline[2pt/2pt]
& & & \frac{\beta}{c^2}I_{s\times s} &  &  \\
& & & & &  \\ \hdashline[2pt/2pt]
&  &L_{l\times n}&  &  & \frac{\beta}{c^2}I_{l\times l}
\end{array}\right)
& \begin{array}{l}
\left.\rule{0mm}{7mm}\right\}M\\
\\\left.\rule{0mm}{7mm}\right\}N
\end{array}\\[-5pt]
\begin{array}{cc}
\underbrace{\rule{40mm}{0mm}}_M&
\underbrace{\rule{30mm}{0mm}}_N\end{array} &
\end{array},
\]
\[
B=\begin{array}{c@{\hspace{-5pt}}l}
\left(\begin{array}{cc;{2pt/2pt}c;{2pt/2pt}cc;{2pt/2pt}c}
C_{m\times m}
&  &  &  &  &  \\ \hdashline[2pt/2pt]
&  &  &  &  \\
& & C_{n\times n}& & &E_{n\times l}  \\ \hdashline[2pt/2pt]
& & & D_{s\times s} &  &  \\
& & & & &  \\ \hdashline[2pt/2pt]
&  &&  &  & D_{l\times l}
\end{array}\right)
& \begin{array}{l}
\left.\rule{0mm}{7mm}\right\}M\\
\\\left.\rule{0mm}{7mm}\right\}N
\end{array}\\[-5pt]
\begin{array}{cc}
\underbrace{\rule{28mm}{0mm}}_M&
\underbrace{\rule{26mm}{0mm}}_N\end{array} &
\end{array},
\]
\[
\begin{array}{cc}
{\bf U}_{\overline{M}}=&({u}_{M1}^{i},\cdots,{u}_{Mm}^{i};
{u}_{M(m+1)}^{e},\cdots, {u}_{MM}^{e};
{p}_{N1}^{i},\cdots, {p}_{Ns}^{i};{p}_{N(s+1)}^{e},\cdots,
{p}_{NN}^{e})^{\top}\\
&\underbrace{\rule{27mm}{0mm}}_m
\underbrace{\rule{33mm}{0mm}}_n
\underbrace{\rule{23mm}{0mm}}_s
\underbrace{\rule{32mm}{0mm}}_l
\end{array} ,
\]
\[
\begin{array}{cc}
F=&(0,\cdots,0;0,\cdots, 0;
f^{1},\cdots, f^{s};0,\cdots, 0)^{\top}\\
&\underbrace{\rule{14mm}{0mm}}_m
\underbrace{\rule{14mm}{0mm}}_n
\underbrace{\rule{18mm}{0mm}}_s
\underbrace{\rule{14mm}{0mm}}_l
\end{array},
\]
Since $A$ is invertible, it follows from the standard theory of ordinary
differential equations that there exists a unique $C^2$ function ${\bf
U}_{\overline{M}}(t)$ consisting of ${\rm p}_N$ and ${\bf u}_M$ which satisfy
\eqref{puic}--\eqref{vf6} for $t\in[0, T]$.
\end{proof}

Define two product spaces
\begin{align*}
\mathcal{H}^{1}:&=\boldsymbol{H}^{1}(D)\times \boldsymbol{L}^{2}(\partial D)
\times H^1(\Omega)\times L^{2}(\partial\Omega),\\
\mathcal{L}^{2}:&=\boldsymbol{L}^{2}(D)\times \boldsymbol{L}^{2}(\partial D)
\times L^2(\Omega)\times L^{2}(\partial\Omega).
\end{align*}
Let ${U}=({\bf u}^{i};{\bf u}^{e}; {\rm p}^{i};{\rm p}^{e})^{\top}$. The
norms of $U$ in $\mathcal H^1$ and $\mathcal L^2$ are defined by
\begin{align*}
\|U\|_{\mathcal{H}^{1}}^2 &=\|{\bf u}^{i}\|_{\boldsymbol{H}^{1}(D)}^2+\|{\bf
u}^{e}\|_{\boldsymbol{L}^{2}(\partial D)}^2+
\|{\rm p}^{i}\|_{H^1(\Omega)}^2+\|{\rm p}^{e}\|_{L^{2}(\partial\Omega)}^2,\\
\|U\|_{\mathcal{L}^{2}}^2 &=\|{\bf u}^{i}\|_{\boldsymbol{L}^{2}(D)}^2+\|{\bf
u}^{e}\|_{\boldsymbol{L}^{2}(\partial D)}^2+ \|{\rm
p}^{i}\|_{L^2(\Omega)}^2+\|{\rm p}^{e}\|_{L^2(\partial\Omega)}^2.
\end{align*}
Let $\|\cdot\|_{op}$ denote the operator norm.

\begin{theorem}\label{s2-t3}
 There exists a positive constant $C$ depending only on $\Omega, D, T,$ and the
coefficients of the acoustic-elastic interaction problem \eqref{ibvpo} such that
 \begin{align*}
&  \max_{t\in[0, T]}\left(\|{\bf
u}'_{\overline{M}}(t)\|^2_{\mathcal{L}^2}+\|{\bf
u}_{\overline{M}}(t)\|^2_{\mathcal{H}^1}
 \right)+\|{\bf u}''_{\overline{M}}(t)\|^{2}_{L^2(0, T;\mathcal{H}^{-1})}\\
&\qquad\leq C\left(\|f\|^{2}_{L^2(0, T; L^2(\Omega))}+\|
g\|^{2}_{H^1(\Omega)}+\|h\|^{2}_{L^2(\Omega)} \right),\quad \overline{M}=1, 2,
\dots,
 \end{align*}
 where ${\bf u}_{\overline{M}}=({\bf u}_M, {\rm p}_N)^{\top}$.
\end{theorem}

\begin{proof}
It is easy to see that the lower triangular matrix $A$ has a bounded inverse
$A^{-1}$. We have from \eqref{ode} that
\begin{equation}\label{ode1}
 {\bf U}''_{\overline{M}}+A^{-1}B{\bf U}_{\overline{M}}=A^{-1}F.
\end{equation}
Taking the inner product with ${\bf U}'_{\overline{M}}$ on both sides
of \eqref{ode1} yields
\begin{equation}\label{s2-t3-1}
({\bf U}''_{\overline{M}},{\bf U}'_{\overline{M}})+(A^{-1}B{\bf
U}_{\overline{M}},{\bf U}'_{\overline{M}})=(A^{-1}F,{\bf U}'_{\overline{M}})
\quad\text{for a.e.} ~t\in[0, T].
\end{equation}
Observe that
\begin{equation}\label{s2-t3-2}
 ({\bf U}''_{\overline{M}},{\bf U}'_{\overline{M}})=\frac{\rm d}{{\rm
d}t}\left(\frac{1}{2}\|{\bf U}'_{\overline{M}}\|^2 \right).
\end{equation}
Combining \eqref{s2-t3-1}--\eqref{s2-t3-2} and using the Cauchy--Schwarz
inequality, we obtain
\begin{align}\label{s2-t3-3}
\frac{1}{2}\frac{\rm d}{{\rm
d}t}\|{\bf U}'_{\overline{M}}\|^2
 &\leq |(A^{-1}B{\bf U}_{\overline{M}},{\bf U}'_{\overline{M}})|+|(A^{-1}F,{\bf
U}'_{\overline{M}})| \notag\\
 &\leq \frac{1}{2} \left(\|A^{-1}B{\bf U}_{\overline{M}}\|^2+\|{\bf
U}'_{\overline{M}}\|^2+\|A^{-1}F\|^2+\|{\bf U}'_{\overline{M}}\|^2\right)
\notag\\
&\leq  \frac{1}{2}\|A^{-1}B\|^2 \|{\bf U}_{\overline{M}}\|^2+\|{\bf
U}'_{\overline{M}}\|^2+\frac{1}{2}\|A^{-1}F\|^2.
\end{align}
It is clear to note that
\begin{align}\label{s2-t3-4}
\frac{1}{2}\frac{\rm d}{{\rm
d}t}\|{\bf U}_{\overline{M}}\|^2
 \leq |({\bf U}_{\overline{M}},{\bf U}'_{\overline{M}})|
  \leq \frac{1}{2} \left(\|{\bf U}'_{\overline{M}}\|^2+\|{\bf
U}_{\overline{M}}\|^2\right).
\end{align}
Using \eqref{s2-t3-3} and \eqref{s2-t3-4}, we may consider the inequality
\[
 \alpha'(t)\leq C_1\alpha(t)+\delta(t),\quad t\in[0, T],
\]
where $\alpha(t)=\|{\bf U}'_{\overline{M}}\|^2+\|{\bf U}_{\overline{M}}\|^2$,
$C_1=\max\big\{(1+\|A^{-1}B\|^2_{op}),3\big\}$, $\delta(t)=\|A^{-1}F\|$.
It follows from the Gronwall inequality that
\[
 \alpha(t)\leq e^{C_1 t}\left(\alpha(0)+\int_0^t \delta(s){\rm d}s \right),
\quad
t\in[0, T].
\]
A simple calculation yields
\begin{equation}\label{s2-t3-5}
\alpha(t)\leq e^{C T}\left(\alpha(0)+\|f\|_{L^{2}[0,T;{L^2(\Omega)}]}\right),
\end{equation}
where
\[
 \alpha(0)=\|{\bf U}'_{\overline{M}}(0)\|^2+\|{\bf U}_{\overline{M}}(0)\|^2 \leq
\left(\|h\|^2_{L^2(\Omega)}+\|g\|^2_{H^1(\Omega)}\right).
\]
By Parseval's equality, we have
\[
\|F\|^2=\sum_{k=1}^{m}|(f,w_{k})|^2\leq\|f\|^2_{L^2(\Omega)}
\]
and
\[
\|{\bf U}'_{\overline{M}}(0)\|^2+\|{\bf U}_{\overline{M}}(0)\|^2
=\sum_{k=1}^{m}|(h,w_{k})|^2+\sum_{k=1}^{m}|(g,w_{k})|^2
\leq\|h\|^2_{L^2(\Omega)}+\|g\|^2_{H^1(\Omega)}.
\]
In fact, we may have from straightforward calculations that
\begin{align}\label{Uu}
\|{\bf u}_{\overline{M}}(t)\|^2_{\mathcal{H}^1}&=\int_{D}
\bigg(\sum_{j=1}^m u^i_{Mj}(t) W_j^i\cdot \sum_{k=1}^m u_{Mk}^i(t) W_k^i\bigg)+
\bigg(\sum_{j=1}^m u^i_{Mj}(t) \nabla W_j^i : \sum_{k=1}^m u_{Mk}^i(t)
\nabla W_k^i\bigg){\rm d} x\notag\\
&\quad+\int_{\Omega}
\bigg(\sum_{j=M+1}^{M+s} p^i_{Nj}(t) w_j^i \sum_{k=M+1}^{M+s} p_{Nk}^i(t)
w_k^i\bigg)+\bigg(\sum_{j=M+1}^{M+s} p^i_{Nj}(t) \nabla w_j^i
\cdot\sum_{k=M+1}^{M+s} p_{Nk}^i(t) \nabla w_k^i\bigg) {\rm d} x\notag\\
&\quad+\sum_{j=m+1}^M |u^{e}_{Mj}(t)|^2
\|W_j^e\|^{2}_{\boldsymbol{L}^{2}(\partial
D)} +\sum_{j=M+s+1}^{\overline{M}} |p^{e}_{Nj}(t)|^2
\|w_j^e\|^{2}_{L^{2}(\partial \Omega)}\nonumber \\
&=\sum_{k=1}^{\overline{M}}|{\bf U}^k_{\overline{M}}|^2=\|{\bf
U}_{\overline{M}}\|^2.
\end{align}
Similarly,
\begin{equation}\label{Uu12}
\|{\bf u}_{\overline{M}}^{'}(t)\|^2_{\mathcal{L}^2}=\|{\bf
U'}_{\overline{M}}\|^2.
\end{equation}
Combining \eqref{s2-t3-5}--\eqref{Uu12} leads to
\[
 \|{\bf u}'_{\overline{M}}\|^2_{\mathcal{L}^2}+\|{\bf
u}_{\overline{M}}\|^2_{\mathcal{H}^1}
 \leq C\left(\|g\|^2_{H^1(\Omega)}+\|
h\|^2_{L^2(\Omega)}+\|f\|^2_{L^2(0, T; L^2(\Omega))} \right).
\]
Noting that $t\in[0, T]$ is arbitrary, it follows
\begin{equation}\label{s2-t3-7}
 \max_{t\in[0, T]}\left(\|{\bf u}'_{\overline{M}}\|^2_{\mathcal{L}^2}+\|{\bf
u}_{\overline{M}}\|^2_{\mathcal{H}^1}
 \right)\leq C\left(\|
g\|^2_{H^1(\Omega)}+\|h\|^2_{L^2(\Omega)}+\|f\|^2_{L^2(0, T;
L^2(\Omega))} \right).
\end{equation}

For any ${\bf{v}} \in {\mathcal{H}^{1}}, \|{\bf{v}}\|_{\mathcal{H}^{1}}\leq 1$,
let ${\bf{v}}={\bf{v}}_1+{\bf{v}}_2$, where
${\bf{v}}_1\in {\rm span}\{\widetilde{W}_1, \dots,
\widetilde{W}_{\overline{M}}\}$ and $({\bf{v}}_2, \widetilde{W}_k)=0, k=1,
\dots, {\overline{M}}$, where $\widetilde{W}_k:=(W_k(x),w_k(x))^{\top}$,
$w_k(x)=0  ~ (k=1, \cdots, M)$, $W_k(x)=0 ~ (k=M+1, \cdots,
\overline{M})$. Note that
\[
 {\bf v}_1(t)={\displaystyle \sum_{j=1}^{\overline{M}} v_{\overline{M}j}(t)
\widetilde{W}_j},\quad \|{\bf{v}}_1\|_{\mathcal{H}^{1}}\leq 1.
\]
Let ${\bf {V}}_{\overline{M}}=(v_{\overline{M}1},\cdots
v_{\overline{M}\overline{M}}).$  By the definition of the operator norm
\begin{equation}\label{dual}
 \|{\bf u}''_{\overline{M}}\|_{\mathcal{H}^{-1}}=\sup_{\|{\bf v}\|=1}\langle
{\bf u}''_{\overline{M}},{\bf v} \rangle=\sup_{\|{\bf v}\|=1}\langle {\bf
u}''_{\overline{M}},{\bf v}_{1} \rangle.
\end{equation}
It follows from \eqref{pmt}, \eqref{umt}, and \eqref{ode1} that
\begin{align*}
\langle {\bf u}''_{\overline{M}},{\bf v} \rangle &=({\bf
u}''_{\overline{M}},{\bf
v})_{\mathcal{H}^{1}}=({\bf u}''_{\overline{M}},{\bf
v}_{1})_{\mathcal{H}^{1}}\notag\\
&=({\bf U}''_{\overline{M}},{\bf {V}}_{\overline{M}})
\leq C\left(\|f\|_{L^2(\Omega)}+\|{\bf U}_{\overline{M}}\| \right)
\quad\text{for a.e.} ~t\in[0, T].
\end{align*}
In fact, we may easily verify that
\begin{align}\label{uv}
({\bf u}''_{\overline{M}},{\bf v}_{1})_{\mathcal{H}^{1}}
&=\bigg(\sum_{l=1}^{\overline{M}} u_{\overline{M}l}(t) \widetilde{W}_l,\,
\sum_{j=1}^{\overline{M}} v_{\overline{M}j}(t)
\widetilde{W}_j\bigg)_{\mathcal{H}^{1}} =({\bf U}''_{\overline{M}},{\bf
{V}}_{\overline{M}})\\ \nonumber
& \leq  \|{\bf U}''_{\overline{M}}\| \|{\bf {V}}_{\overline{M}}\| \leq
\|A^{-1}F-A^{-1}B{\bf U}_{\overline{M}}\|
\leq C\left(\|f\|_{L^2(\Omega)}+\|{\bf U}_{\overline{M}}\| \right).
\end{align}
Following from \eqref{Uu}, \eqref{dual}, and \eqref{uv} gives
\[
 \|{\bf u}''_{\overline{M}}\|_{\mathcal{H}^{-1}}\leq
C\left(\|f\|_{L^2(\Omega)}+\|{\bf u}_{\overline{M}}\|_{\mathcal{H}^{1}} \right).
\]
Hence,
\begin{align}\label{s2-t3-9}
 \int_0^T \|{\bf u}''_{\overline{M}}\|^2_{\mathcal{H}^{-1}}{\rm d}t&\leq
C\int_0^T\left(\|f\|_{L^2(\Omega)}+\|{\bf u}_{\overline{M}}\|_{\mathcal{H}^{1}}
\right)^2{\rm d}t\notag\\
&\leq C\left(\|g\|^2_{H^1(\Omega)}+\|h\|^2_{L^2(\Omega)}+\|f\|^2_{L^2(0, T;
L^2(\Omega))}\right).
\end{align}
The proof is completed after combining \eqref{s2-t3-7} and \eqref{s2-t3-9}.
\end{proof}

Now we pass to limits in the Galerkin approximations to obtain the existence of
a weak solution.

\begin{theorem}
There exists a weak solution of the initial boundary value problem \eqref{ibvp}.
\end{theorem}

\begin{proof}
It follows from the energy estimate in Theorem \ref{s2-t3} that
\begin{align*}
&\{{\bf u}_{\overline{M}}\}_{M,N=1}^\infty ~\text{is bounded in}~ L^2(0, T;
{\mathcal{H}^{1}}),\\
&\{{\bf u}_{\overline{M}}^{'}\}_{M,N=1}^\infty ~\text{is bounded in}~ L^2(0, T;
{\mathcal{L}^{2}}),\\
&\{{\bf u}_{\overline{M}}^{''}\}_{M,N=1}^\infty ~\text{is bounded in}~ L^2(0, T;
{\mathcal{H}^{-1}}).
\end{align*}
Therefore, there exists a subsequence still denoted as $\{{\bf
u}_{\overline{M}}\}_{M,N=1}^\infty$ and ${\bf u}\in L^2(0, T;
{\mathcal{H}^{1}})$
with ${\bf u}^{'}\in L^2(0, T; {\mathcal{L}^{2}})$ and ${\bf u}^{''}\in L^2(0,
T; {\mathcal{H}^{-1}})$ such that
\begin{equation}\label{s2-t4-1}
 \begin{cases}
{\bf u}_{\overline{M}} \rightharpoonup{\bf u}\quad&\text{weakly in}~ L^2(0, T;
{\mathcal{H}^{1}}),\\
{\bf u}_{\overline{M}}^{'}  \rightharpoonup{\bf u}^{'}\quad&\text{weakly in}~
L^2(0, T; {\mathcal{L}^{2}}),\\
{\bf u}_{\overline{M}}^{''}  \rightharpoonup{\bf u}^{''}\quad&\text{weakly in}~
L^2(0, T; {\mathcal{H}^{-1}}),
 \end{cases}
\end{equation}
which imply
\[
 \begin{cases}
({\bf u}_M, {\rm p}_N )^{\top} \rightharpoonup ({\bf u}, {\rm p})^{\top},\\
({\bf u}_M^{'}, {\rm p}_N^{'})^{\top}  \rightharpoonup ({\bf u}^{'}, {{\rm
p}}^{'})^{\top},\\
({\bf u}_M^{''}, {\rm p}_N^{''})^{\top}  \rightharpoonup ({\bf u}^{''}, {{\rm
p}}^{''})^{\top}.
 \end{cases}
\]

Next we fix integers $N_{1}, N_{2}$ and choose functions ${\rm q}\in C^1([0,
T]; H_{0}^{1}(\Omega)\times L^{{2}}(\partial\Omega))$ and
${\bf v}\in C^1([0, T]; \boldsymbol{H}_{0}^{1}(D)\times
\boldsymbol{L}^{{2}}(\partial D))$ of the form
\begin{equation}\label{s2-t4-2}
{\rm q}(t)=\sum_{k=1}^{N_{1}} q_{N_{1} k}(t) w_k,\quad  {\bf
v}(t)=\sum_{j=1}^{N_{2}} v_{N_{2} j}(t) {W}_{j}
\end{equation}
where $q_{N_{1} k}, v_{N_{2} j}, k=1, \dots, N_{1}, j=1, \dots, N_{2}$ are
smooth functions. Letting $m\geq \max\{N_{1}, N_{2}\}$ where $m=\min\{M,N\}$,
we have from \eqref{vf6} that
\begin{align}\label{s2-t4-3}
\int_{0}^{T} & \Big(\frac{\beta^{2}}{c}\langle {{\rm p}_{m}^{''},{\rm
q}}\rangle+\rho_{1}\rho_{2}\langle {{\bf u}_{m}^{''},{\bf v}}\rangle
 +a_{0}[{\rm p}_{m},{\rm q};t] +\rho_{1}a_{1}[{\bf u}_{m},{\bf v};t]\notag\\
&\quad +\rho_{1}\int_{\partial D}({\rm p}_{m}n_{D}\cdot{\bf v}-(n_{D}\cdot{\bf
u}_{m}^{''}){\rm q}){\rm d}s\Big) {\rm d}t
=\int_{0}^{T}({\rm f},{\rm q}){\rm d}t.
\end{align}
Using \eqref{s2-t4-1} and taking the limits $m\to\infty$ in \eqref{s2-t4-3}
yields
\begin{align}\label{s2-t4-4}
\int_{0}^{T} & \Big(\frac{\beta^{2}}{c}\langle {{\rm p}^{''},{\rm
q}}\rangle+\rho_{1}\rho_{2}\langle {{\bf u}^{''},{\bf v}}\rangle
 +a_{0}[{\rm p},{\rm q};t]+\rho_{1}a_{1}[{\bf u},{\bf v};t]\notag\\
&\quad +\rho_{1}\int_{\partial D}({\rm p} n_{D}\cdot{\bf v}-(n_{D}\cdot{\bf
u}^{''}){\rm q}){\rm d}s\Big) {\rm d}t
=\int_{0}^{T}({\rm f},{\rm q}){\rm d}t,
\end{align}
which holds for any function ${\rm q}\in L^2([0, T];
\widetilde{H}_{0}^{1}(\Omega))$ and
${\bf v}\in L^2([0, T]; \boldsymbol{H}^{1}(D))$ since
functions of the form \eqref{s2-t4-2} are dense in the space.
Moreover, we have from \eqref{s2-t4-4} that for any ${\rm
\widetilde{q}}\in H_{0}^{1}(\Omega),
{\bf \widetilde{v}}\in \boldsymbol{H}^{1}(D)$ and $t\in[0,T]$
\begin{align*}
 \frac{\beta^{2}}{c} & \langle {{\rm p}^{''},{\rm
\widetilde{q}}}\rangle+\rho_{1}\rho_{2}\langle {{\bf u}^{''},{\bf
\widetilde{v}}}\rangle
 +a_{0}[{\rm p},{\rm \widetilde{q}};t]+\rho_{1}a_{1}[{\bf u},{\bf
\widetilde{v}};t]\\
&\quad  +\rho_{1}\int_{\partial D}({\rm p} n_{D}\cdot{\bf
\widetilde{v}}-(n_{D}\cdot{\bf u}^{''}){\rm \widetilde{q}}){\rm d}s
=({\rm f}, {\rm \widetilde{q}})
\end{align*}
and
\begin{align*}
& {\rm p}\in C(0,T; L^{2}(\Omega)),\quad {\rm p}^{'}\in C(0,T;
H^{-1}(\Omega)),\\
& {\bf u}\in C(0,T; \boldsymbol{L}^{2}(D)),\quad{\bf u}^{'}\in C(0,T;
\boldsymbol{H}^{-1}(D)).
\end{align*}

Next is to verify
\begin{equation}\label{s2-t4-5}
{\rm p}|_{t=0}=g, \quad {\rm p}'|_{t=0}=h.
\end{equation}
Choose any function ${\rm q}\in C^2([0, T]; \widetilde{H}_{0}^{1}(\Omega))$ with
${\rm q}(T)={\rm q}^{'}(T)=0$
and ${\bf v}\in C^2([0, T]; \boldsymbol{H}^{1}(D))$ with ${\bf v}(T)={\bf
v}^{'}(T)=0={\bf v}(0)={\bf v}^{'}(0)$.
Using the integration by parts twice with respect to $t$ in \eqref{s2-t4-4}
gives
\begin{align}\label{s2-t4-6}
\int_{0}^{T} &\Big(\frac{\beta^{2}}{c}( {{\rm q}^{''},{\rm
p}})+\rho_{1}\rho_{2}( {{\bf v}^{''},{\bf u}})
+a_{0}[{\rm p},{\rm q};t]+\rho_{1}a_{1}[{\bf u},{\bf v};t]\notag\\
&\quad +\rho_{1}\int_{\partial D}({\rm p} n_{D}\cdot{\bf v}-(n_{D}\cdot{\bf
u}){\rm q}^{''}){\rm d}s\Big) {\rm d}t
=\int_{0}^{T}({\rm f},{\rm q}){\rm d}t-({\rm p}(0),{\rm q}^{'}(0))+\langle
{{\rm p}^{'}(0),{\rm q}(0)}\rangle.
\end{align}
Similarly, we have from \eqref{s2-t4-3} that
\begin{align}\label{s2-t4-7}
\int_{0}^{T} & \Big(\frac{\beta^{2}}{c}( {{\rm q}^{''},{\rm
p}_{m}})+\rho_{1}\rho_{2}( {{\bf v}^{''},{\bf u}_{m}})
+a_{0}[{\rm p}_{m},{\rm q};t]+\rho_{1}a_{1}[{\bf u}_{m},{\bf v};t]\notag\\
&\qquad +\rho_{1}\int_{\partial D}({\rm p}_{m} n_{D}\cdot{\bf v}-(n_{D}\cdot{\bf
u}_{m}){\rm q}^{''}){\rm d}s \Big) {\rm d}t\notag\\
&\quad =\int_{0}^{T}({\rm f},{\rm q}){\rm d}t-({\rm p}_{m}(0),{\rm
q}^{'}(0))+\langle
{{\rm p}_{m}^{'}(0),{\rm q}(0)}\rangle.
\end{align}
Taking the limits $m\to\infty$ in \eqref{s2-t4-7}, using \eqref{puic} and
\eqref{s2-t4-1}, we get
\begin{align}\label{s2-t4-8}
\int_{0}^{T}&\Big(\frac{\beta^{2}}{c}( {{\rm q}^{''},{\rm
p}})+\rho_{1}\rho_{2}( {{\bf v}^{''},{\bf u}})
+a_{0}[{\rm p},{\rm q};t]+\rho_{1}a_{1}[{\bf u},{\bf v};t]\notag\\
&\qquad+\rho_{1}\int_{\partial D}({\rm p} n_{D}\cdot{\bf v}-(n_{D}\cdot{\bf
u}){\rm q}^{''}){\rm d}s\Big) {\rm d}t\notag\\
&=\int_{0}^{T}({\rm f},{\rm q}){\rm d}t-(g,{\rm q}^{'}(0))+\langle {h,{\rm
q}(0)}\rangle.
\end{align}
Comparing \eqref{s2-t4-6} and \eqref{s2-t4-8}, we conclude \eqref{s2-t4-5}
since ${\rm q}(0)$ and ${\rm q}^{'}(0)$ are arbitrary. Hence (${\rm p}, {\bf
u})$ is a weak solution of the initial boundary value problem \eqref{ibvp}.
\end{proof}

Taking the partial derivatives of \eqref{eq2}, \eqref{iv2},
and the second term of \eqref{bv2} with respect to $t$, we consider
\begin{subequations} \label{ibvp2}
\begin{numcases}
{ }
\frac{\beta}{c^2}\partial_{t}^{2} p-\nabla\cdot(M\nabla p)=f \label{eq1n}
&\text{in} ~ $\Omega\times(0, T]$,  \\
p=0 \label{bv1n} &\text{on}~ $\partial B_b\times (0, T]$,\\
p|_{t=0}=g,\quad\partial_t p|_{t=0}=h \label{iv1n} &in $\Omega$\\
\mu\Delta(\partial_{t}\boldsymbol{u})+(\lambda+\mu)\nabla\nabla\cdot(\partial_{t
}\boldsymbol{u}) -\rho_{2}\partial_t^{2}(\partial_{t}\boldsymbol{u})=0
\label{eq2n} &\text{in} ~ $D\times (0, T]$, \\
(\partial_{t}\boldsymbol{u})|_{t=0}=0,\quad
\partial_{t}^{2} \boldsymbol{u}|_{t=0}=\rho_{2}^{-1}(\mu\Delta\boldsymbol{u}
+(\lambda+\mu)\nabla\nabla\cdot\boldsymbol{u})|_{t=0}=0 \label{iv2n} &\text{in}
~ $D$,\\
\partial_{\boldsymbol{n}_{D}} p=-\rho_{1}\boldsymbol{n}_{D}\cdot\partial_t^{2}
\boldsymbol{u} \label{bv21n}&\text{in} ~  $\partial D\times (0, T]$,\\
-(\partial_{t}p)\boldsymbol{n}_{D}=\mu \partial_{\boldsymbol{n}_{D}}
(\partial_{t}\boldsymbol{u})
+(\lambda+\mu)(\nabla\cdot(\partial_{t}\boldsymbol{u }))\boldsymbol{n}_{D}
\label{bv2n}&\text{in} ~  $\partial D\times (0, T]$.
\end{numcases}
\end{subequations}

\begin{theorem}
 The initial boundary value problem \eqref{ibvp2} has a unique weak solution.
\end{theorem}

\begin{proof}
It suffices to show that $p=0, \boldsymbol u=0$ if $f=g=h=0$.
Fix $0\leq t\leq T$ and let
\[
E(t):=E_{1}(t)+E_{2}(t),
\]
where
\begin{align*}
&E_{1}(t)=\|\frac{\sqrt{\beta}}{c}\partial_{t}
p\|^{2}_{L^{2}(\Omega)}+\|M^{\frac{1}{2}}\nabla
p\|^{2}_{L^{2}(\Omega)},\nonumber\\
&E_{2}(t)=\|\sqrt{\rho_{1}\rho_{2}}\ \partial_{t}^{2}
\boldsymbol{u}\|^{2}_{\boldsymbol{L}^{2}(D)}
+\|\sqrt{\rho_{1}(\lambda+\mu)}\ \nabla\cdot(\partial_{t}
\boldsymbol{u})\|^{2}_{L^{2}(D)}
+\|\sqrt{\rho_{1}\mu}\ \nabla(\partial_{t}
\boldsymbol{u})\|^{2}_{L^{2}(D)^{d\times d}}.
\end{align*}
Then for each $t\in[0, T]$, we have
\begin{equation}\label{s2-t5-1}
 E(t)-E(0)=\int_0^t E^{'}(\tau){\rm d}\tau=\int_0^t E_{1}^{'}(\tau){\rm
d}\tau+\int_0^t E_{2}^{'}(\tau){\rm d}\tau.
\end{equation}
Following from \eqref{ibvp2} and the integration by parts, we obtain
\begin{align}\label{s2-t5-2}
\int_{0}^{t} E_{1}^{'}(\tau){\rm d}\tau
&=2\int_{0}^{t}\int_{\Omega}\left(\frac{\beta}{c^{2}}(\partial_{\tau}^{2}
p)(\partial_{\tau} p)
+(M^{\frac{1}{2}}\nabla (\partial_{\tau}p))\cdot (M^{\frac{1}{2}}\nabla
p)\right) {\rm d}x {\rm d}\tau\nonumber\\
&=2\int_{0}^{t}\int_{\Omega}\left((\partial_{\tau}p)(\nabla\cdot(M\nabla p))
+(M^{\frac{1}{2}}\nabla (\partial_{\tau}p))\cdot (M^{\frac{1}{2}}\nabla p)
+(\partial_{\tau}p)f\right) {\rm d}x {\rm d}\tau\nonumber\\
&=2\int_{0}^{t}\int_{\Omega}\left(-(M^{\frac{1}{2}}\nabla
(\partial_{\tau}p))\cdot (M^{\frac{1}{2}}\nabla p)
+(M^{\frac{1}{2}}\nabla (\partial_{\tau}p))\cdot (M^{\frac{1}{2}}\nabla p)
+(\partial_{\tau}p)f\right) {\rm d}x {\rm d}\tau\nonumber\\
&\quad+2\int_{0}^{t}\int_{\partial B_{b}}0 {\rm d}s{\rm d}\tau
-2\int_{0}^{t}\int_{\partial D}(\partial_{\boldsymbol{n}_{D}}p)
(\partial_{\tau}p) {\rm d}s {\rm d}\tau\nonumber\\
&=2\int_{0}^{t}\rho_{1}\int_{\partial
D}(\boldsymbol{n}_{D}\cdot\partial_{\tau}^{2}\boldsymbol{u})(\partial_{\tau}p)
{\rm d}s {\rm d}\tau
\end{align}
and
\begin{align}\label{s2-t5-3}
\int_{0}^{t} E_{2}^{'}(\tau){\rm d}\tau
&=2\int_{0}^{t}\int_{D}\left(\rho_{1}\rho_{2}(\partial_{\tau}^{3}
\boldsymbol{u})\cdot(\partial_{\tau}^{2} \boldsymbol{u})
+\rho_{1}(\lambda+\mu)[\nabla\cdot(\partial_{\tau}^{2} \boldsymbol{u})]
[\nabla\cdot(\partial_{\tau}\boldsymbol{u})]\right) {\rm d}x {\rm
d}\tau\nonumber\\
&\quad+2\int_{0}^{t}\int_{D}\left(\rho_{1}\mu[\nabla(\partial_{\tau}^{2}
\boldsymbol{u})]:[\nabla(\partial_{\tau} \boldsymbol{u})]\right) {\rm d}x {\rm
d}\tau\nonumber\\
&=2\int_{0}^{t}\int_{D}\left(\rho_{1}\mu(\Delta(\partial_{\tau}\boldsymbol{u}
))\cdot(\partial_{\tau}^{2} \boldsymbol{u})
+\rho_{1}(\lambda+\mu)(\nabla\nabla(\partial_{\tau}\boldsymbol{u}
))\cdot(\partial_{\tau}^{2} \boldsymbol{u})\right) {\rm d}x {\rm
d}\tau\nonumber\\
&\quad+2\int_{0}^{t}\int_{D}\left(\rho_{1}(\lambda+\mu)[\nabla\cdot(\partial_{
\tau}^{2} \boldsymbol{u})]
[\nabla\cdot(\partial_{\tau}\boldsymbol{u})]
+\rho_{1}\mu[\nabla(\partial_{\tau}^{2} \boldsymbol{u})]:[\nabla(\partial_{\tau}
\boldsymbol{u})]\right) {\rm d}x {\rm d}\tau\nonumber\\
&=2\int_{0}^{t}\int_{D}\left(-\rho_{1}\mu[\nabla(\partial_{\tau}^{2}
\boldsymbol{u})]:[\nabla(\partial_{\tau} \boldsymbol{u})]
-\rho_{1}(\lambda+\mu)[\nabla\cdot(\partial_{\tau}^{2} \boldsymbol{u})]
[\nabla\cdot(\partial_{\tau}\boldsymbol{u})]\right) {\rm d}x {\rm
d}\tau\nonumber\\
&\quad+2\int_{0}^{t}\int_{D}\left(\rho_{1}(\lambda+\mu)[\nabla\cdot(\partial_{
\tau}^{2} \boldsymbol{u})]
[\nabla\cdot(\partial_{\tau}\boldsymbol{u})]
+\rho_{1}\mu[\nabla(\partial_{\tau}^{2} \boldsymbol{u})]:[\nabla(\partial_{\tau}
\boldsymbol{u})]\right) {\rm d}x {\rm d}\tau\nonumber\\
&\quad+2\int_{0}^{t}\int_{\partial
D}\rho_{1}\left(\mu\partial_{\boldsymbol{n}_{D}}(\partial_{\tau}\boldsymbol{u}
)\cdot(\partial_{\tau}^{2}\boldsymbol{u})
+(\lambda+\mu)(\nabla\cdot(\partial_{\tau}\boldsymbol{u})\boldsymbol{n}_{D}
)\cdot(\partial_{\tau}^{2}\boldsymbol{u})\right){\rm d}x {\rm d}\tau\nonumber\\
&=-2\int_{0}^{t}\int_{\partial
D}\rho_{1}(\partial_{\tau}p)(\boldsymbol{n}_{D}\cdot\partial_{\tau}^{2}
\boldsymbol{u}) {\rm d}s {\rm d}\tau.
\end{align}
It is easy to note that if $f=g=h=0$, we have
\[
E(0)=0.
\]
Thus, combining \eqref{s2-t5-1}--\eqref{s2-t5-3}, we obtain
\[
 E(t)= E_{1}(t)+E_{2}(t)=0,
\]
which implies that
\[
 \partial_{t}p=\nabla p=
\partial_{t}^{2}\boldsymbol{u}=\nabla\cdot(\partial_{t}\boldsymbol{u}
)=\nabla(\partial_{t}\boldsymbol{u})=0.
\]
Thus we obtain from initial conditions in \eqref{ibvp2} that $p=0,
\boldsymbol u=0$ if $f=g=h=0$, which completes the proof.
\end{proof}

\subsection{Stability}

In this section we discuss the stability estimate for the unique weak solution
of the initial boundary value problem \eqref{ibvp2}.

\begin{theorem}\label{s3-t1}
Let $(p, \boldsymbol u)$  be the unique weak solution of the initial boundary
value problem \eqref{ibvp2}.
Given $f\in L^2(\Omega), g\in \widetilde{H}_0^1(\Omega), h\in L^2(\Omega)$,
there exists a positive constant $C$ such that
\begin{align*}
&\max \limits_{t \in [0, T]} \Big\{\|\partial_t  p(\cdot, t)
\|^{2}_{L^2(\Omega)} +\|\nabla p(\cdot, t)
\|^{2}_{\boldsymbol{L}^{2}(\Omega)}\notag\\
&\quad\quad+\|\partial_{t}^{2}  \boldsymbol{u} (\cdot, t)
\|^{2}_{\boldsymbol{L}^{2}(D)} +\|\nabla\cdot(\partial_t  \boldsymbol{u} (\cdot,
t)) \|^{2}_{L^2(D)} +\|\nabla(\partial_t  \boldsymbol{u} (\cdot, t))
\|^{2}_{L^2(D)^{d\times d}}\Big\} \notag\\
&\leq C \left( \|f\|^{2}_{L^{1}(0, T; L^2(\Omega))}
+\|g\|^{2}_{H^{1}(\Omega)}
+\|h\|^{2}_{L^{2}(\Omega)}\right).
\end{align*}
\end{theorem}

\begin{proof}
It follows from the discussion in previous section that the initial boundary
value problem \eqref{ibvp} has a unique weak solution $({p}, \boldsymbol{u})$
satisfying
\begin{align*}
p &\in L^2(0, T; {\widetilde{H}_{0}^{1}}(\Omega))\cap H^{1} (0, T;
L^2(\Omega)),\\
\boldsymbol{u} &\in L^2(0, T; \boldsymbol{H}^{1}(D))\cap H^{1} (0, T;
\boldsymbol{L}^2(D)).
\end{align*}
For any $t\in[0, T]$, consider the energy function
\[
E(t):=E_{1}(t)+E_{2}(t),
\]
where
\begin{align*}
&E_{1}(t)=\|\frac{\sqrt{\beta}}{c}\partial_{t}
p\|^{2}_{L^{2}(\Omega)}+\|M^{\frac{1}{2}}\nabla
p\|^{2}_{L^{2}(\Omega)},\nonumber\\
&E_{2}(t)=\|\sqrt{\rho_{1}\rho_{2}}\ \partial_{t}^{2}
\boldsymbol{u}\|^{2}_{\boldsymbol{L}^{2}(D)}
+\|\sqrt{\rho_{1}(\lambda+\mu)}\ \nabla\cdot(\partial_{t}
\boldsymbol{u})\|^{2}_{L^{2}(D)}
+\|\sqrt{\rho_{1}\mu}\ \nabla(\partial_{t}
\boldsymbol{u})\|^{2}_{L^{2}(D)^{d\times d}}.
\end{align*}
Then for each $t\in[0, T]$, we have
\begin{equation}\label{s3-t1-1}
 E(t)-E(0)=\int_0^t E^{'}(\tau){\rm d}\tau=\int_0^t E_{1}^{'}(\tau){\rm
d}\tau+\int_0^t E_{2}^{'}(\tau){\rm d}\tau.
\end{equation}
By \eqref{ibvp2} and the integration by parts, we obtain
\begin{align}\label{s3-t1-2}
\int_{0}^{t} E_{1}^{'}(\tau){\rm d}\tau
&=2\int_{0}^{t}\int_{\Omega}\left(\frac{\beta}{c^{2}}(\partial_{\tau}^{2}
p)(\partial_{\tau} p)
+(M^{\frac{1}{2}}\nabla (\partial_{\tau}p))\cdot (M^{\frac{1}{2}}\nabla
p)\right) {\rm d}x {\rm d}\tau\nonumber\\
&=2\int_{0}^{t}\int_{\Omega}\left((\partial_{\tau}p)(\nabla\cdot(M\nabla p))
+(M^{\frac{1}{2}}\nabla (\partial_{\tau}p))\cdot (M^{\frac{1}{2}}\nabla p)
+(\partial_{\tau}p)f\right) {\rm d}x {\rm d}\tau\nonumber\\
&=2\int_{0}^{t}\int_{\Omega}\left(-(M^{\frac{1}{2}}\nabla
(\partial_{\tau}p))\cdot (M^{\frac{1}{2}}\nabla p)
+(M^{\frac{1}{2}}\nabla (\partial_{\tau}p))\cdot (M^{\frac{1}{2}}\nabla p)
+(\partial_{\tau}p)f\right) {\rm d}x {\rm d}\tau\nonumber\\
&\quad+2\int_{0}^{t}\int_{\partial B_{b}}0 {\rm d}s {\rm d}\tau
-2\int_{0}^{t}\int_{\partial D}(\partial_{\boldsymbol{n}_{D}}p)
(\partial_{\tau}p) {\rm d}s {\rm d}\tau\nonumber\\
&=2\int_{0}^{t}\rho_{1}\int_{\partial
D}(\boldsymbol{n}_{D}\cdot\partial^2_{\tau}\boldsymbol{u})(\partial_{\tau}p)
{\rm d}s {\rm d}\tau
+2\int_{0}^{t}\int_{\Omega}(\partial_{\tau}p)f {\rm d}x {\rm d}\tau,
\end{align}
and
\begin{align}\label{s3-t1-3}
\int_{0}^{t} E_{2}^{'}(\tau){\rm d}\tau
&=2\int_{0}^{t}\int_{D}\left(\rho_{1}\rho_{2}(\partial_{\tau}^{3}
\boldsymbol{u})\cdot(\partial_{\tau}^{2} \boldsymbol{u})
+\rho_{1}(\lambda+\mu)[\nabla\cdot(\partial_{\tau}^{2} \boldsymbol{u})]
[\nabla\cdot(\partial_{\tau}\boldsymbol{u})]\right) {\rm d}x {\rm
d}\tau\nonumber\\
&\quad+2\int_{0}^{t}\int_{D}\left(\rho_{1}\mu[\nabla(\partial_{\tau}^{2}
\boldsymbol{u})]:[\nabla(\partial_{\tau} \boldsymbol{u})]\right) {\rm d}x {\rm
d}\tau\nonumber\\
&=2\int_{0}^{t}\int_{D}\left(\rho_{1}\mu(\Delta(\partial_{\tau}\boldsymbol{u}
))\cdot(\partial_{\tau}^{2} \boldsymbol{u})
+\rho_{1}(\lambda+\mu)(\nabla\nabla(\partial_{\tau}\boldsymbol{u}
))\cdot(\partial_{\tau}^{2} \boldsymbol{u})\right) {\rm d}x {\rm
d}\tau\nonumber\\
&\quad+2\int_{0}^{t}\int_{D}\left(\rho_{1}(\lambda+\mu)[\nabla\cdot(\partial_{
\tau}^{2} \boldsymbol{u})]
[\nabla\cdot(\partial_{\tau}\boldsymbol{u})]
+\rho_{1}\mu[\nabla(\partial_{\tau}^{2} \boldsymbol{u})]:[\nabla(\partial_{\tau}
\boldsymbol{u})]\right) {\rm d}x {\rm d}\tau\nonumber\\
&=2\int_{0}^{t}\int_{D}\left(-\rho_{1}\mu[\nabla(\partial_{\tau}^{2}
\boldsymbol{u})]:[\nabla(\partial_{\tau} \boldsymbol{u})]
-\rho_{1}(\lambda+\mu)[\nabla\cdot(\partial_{\tau}^{2} \boldsymbol{u})]
[\nabla\cdot(\partial_{\tau}\boldsymbol{u})]\right) {\rm d}x {\rm
d}\tau\nonumber\\
&\quad+2\int_{0}^{t}\int_{D}\left(\rho_{1}(\lambda+\mu)[\nabla\cdot(\partial_{
\tau}^{2} \boldsymbol{u})]
[\nabla\cdot(\partial_{\tau}\boldsymbol{u})]
+\rho_{1}\mu[\nabla(\partial_{\tau}^{2} \boldsymbol{u})]:[\nabla(\partial_{\tau}
\boldsymbol{u})]\right) {\rm d}x {\rm d}\tau\nonumber\\
&\quad+2\int_{0}^{t}\int_{\partial
D}\rho_{1}\left(\mu\partial_{\boldsymbol{n}_{D}}(\partial_{\tau}\boldsymbol{u}
)\cdot\partial_{\tau}^{2}\boldsymbol{u}
+(\lambda+\mu)(\nabla\cdot(\partial_{\tau}\boldsymbol{u})\boldsymbol{n}_{D}
)\cdot(\partial_{\tau}^{2}\boldsymbol{u})\right){\rm d}x {\rm d}\tau\nonumber\\
&=-2\int_{0}^{t}\int_{\partial
D}\rho_{1}(\partial_{\tau}p)(\boldsymbol{n}_{D}\cdot\partial_{\tau}^{2}
\boldsymbol{u}) {\rm d}s {\rm d}\tau.
\end{align}

It is easy to note that
\begin{align*}
E(0)&=\|\frac{\sqrt{\beta}}{c}\partial_{t}
p|_{t=0}\|^{2}_{L^{2}(\Omega)}+\|M^{\frac{1}{2}}\nabla
p|_{t=0}\|^{2}_{\boldsymbol{L}^{2}(\Omega)}\nonumber\\
&=\|\frac{\sqrt{\beta}}{c}h\|^{2}_{L^{2}(\Omega)}+\|M^{\frac{1}{2}}\nabla
g\|^{2}_{\boldsymbol{L}^{2}(\Omega)}.
\end{align*}
Combining \eqref{s3-t1-1}--\eqref{s3-t1-3} leads to
\begin{align*}
&\|\frac{\sqrt{\beta}}{c}\partial_{t}
p\|^{2}_{L^{2}(\Omega)}+\|M^{\frac{1}{2}}\nabla
p\|^{2}_{\boldsymbol{L}^{2}(\Omega)},\nonumber\\
&\quad +\|\sqrt{\rho_{1}\rho_{2}}\ \partial_{t}^{2}
\boldsymbol{u}\|^{2}_{\boldsymbol{L}^{2}(D)}
+\|\sqrt{\rho_{1}(\lambda+\mu)}\ \nabla\cdot(\partial_{t}
\boldsymbol{u})\|^{2}_{L^{2}(D)}
+\|\sqrt{\rho_{1}\mu}\ \nabla(\partial_{t} \boldsymbol{u})\|^{2}_{L^{2}(D)^{d
\times d}}\nonumber\\
&=2\int_{0}^{t}\int_{\Omega}(\partial_{\tau}p)f {\rm d}x {\rm d}\tau
 +\|\frac{\sqrt{\beta}}{c}h\|^{2}_{L^{2}(\Omega)}+\|M^{\frac{1}{2}}\nabla
g\|^{2}_{\boldsymbol{L}^{2}(\Omega)}\nonumber\\
 &\leq 2 \max\limits_{t \in [0, T]}\{ \| \partial_t p(\cdot,
t)\|_{L^2(\Omega)}\}\|f\|_{L^1 (0, T; L^2(\Omega))}
+\frac{\beta}{c^{2}}\|h\|^{2}_{L^{2}(\Omega)}+\|M^{\frac{1}{2}}\nabla
g\|^{2}_{\boldsymbol{L}^{2}(\Omega)}.
\end{align*}

Using the Young inequality, we obtain
\begin{align*}
&\max \limits_{t \in [0, T]} \{\|\partial_t  p(\cdot, t) \|^{2}_{L^2(\Omega)}
+\|\nabla p(\cdot, t) \|^{2}_{\boldsymbol{L}^{2}(\Omega)}\notag\\
&\quad\quad+\|\partial_{t}^{2}  \boldsymbol{u} (\cdot, t)
\|^{2}_{\boldsymbol{L}^{2}(D)}
+\|\nabla\cdot(\partial_t  \boldsymbol{u} (\cdot, t)) \|^{2}_{{L}^{2}(D)}
+\|\nabla(\partial_t  \boldsymbol{u} (\cdot, t)) \|^{2}_{{L}^{2}(D)^{d \times
d}} \notag\\
&\leq C\left(\|f\|^{2}_{L^1 (0, T; L^2(\Omega))}
+\frac{\beta}{c^{2}}\|h\|^{2}_{L^{2}(\Omega)}+\|M^{\frac{1}{2}}\nabla
g\|^{2}_{\boldsymbol{L}^{2}(\Omega)}\right)\notag\\
&\leq C \left( \|f\|^{2}_{L^{1}(0, T; L^2(\Omega))}
+\|g\|^{2}_{H^{1}(\Omega)}
+\|h\|^{2}_{L^{2}(\Omega)}\right).
\end{align*}
which completes the proof.
\end{proof}

\subsection{A priori estimate}

In this section we derive an a priori stability estimate for the wave field with
an explicit dependence on the time.

The variational problem of \eqref{ibvp2} is to find $(p, \boldsymbol{u})\in
\widetilde{H}_{0}^{1}(\Omega)\times \boldsymbol{H}^{1}(D)$ for $t\in[0, T]$ such
that
\begin{align}\label{vp}
\int_{\Omega}\frac{\beta}{c^{2}}(\partial_{t}^{2}p) q {\rm d}x
&=-\int_{\Omega}(M^{\frac{1}{2}}\nabla{p})\cdot(M^{\frac{1}{2}}\nabla{q}){\rm
d}x +\int_{\partial D}\rho_{1}(\boldsymbol{n}_{D}\cdot
\partial_t^2\boldsymbol{u})
q{\rm d} s\nonumber\\
&\quad+\int_{\Omega}fq{\rm d}x,\quad\forall\, q\in
\widetilde{H}_{0}^{1}(\Omega),
\end{align}
and
\begin{align}\label{vp1}
\int_{D}\rho_{2}\partial_{t}^{2}(\partial_{t}\boldsymbol{u})\cdot\boldsymbol{v}{
\rm d}x
&=-\int_{D}[(\mu\nabla\partial_{t}\boldsymbol{u}):(\nabla \boldsymbol{v})
+(\lambda+\mu)((\nabla\cdot\partial_{t}\boldsymbol{u})(\nabla\cdot\boldsymbol{v}
))]{\rm d}x\nonumber\\
&\quad-\int_{\partial D}(\partial_{t}\rho)(\boldsymbol{n}_{D}\cdot
\boldsymbol{v}){\rm d} s,\quad \forall\,
\boldsymbol{v}\in\boldsymbol{H}^{1}(D).
\end{align}

\begin{theorem}\label{s3-t7}
Let $\boldsymbol{u}$ be the unique weak solution of the initial boundary value
problem \eqref{ibvp}.
Given $f\in L^{1}[0, T; L^{2}(\Omega)]$, $g, h\in L^{2}(\Omega)$,
there exist positive constants $C_1, C_2$ such that
\begin{align*}
&\|p\|^{2}_{L^{\infty}(0,T; L^{2}(\Omega))}
+\|\nabla p\|^{2}_{L^{\infty}(0,T; \boldsymbol{L}^{2}(\Omega))}\nonumber\\
&\quad +\|\partial_{t}\boldsymbol{u}\|^{2}_{L^{\infty}(0,T;
\boldsymbol{L}^{2}(D))}
+\|\nabla\boldsymbol{u}\|^{2}_{L^{\infty}(0,T; L^{2}(D)^{d\times d})}
+\|\nabla\cdot\boldsymbol{u}\|^{2}_{L^{\infty}(0,T;{L}^{2}(D))}\nonumber\\
&\leq C_{1}\left(\|g\|^{2}_{L^{2}(\Omega)}
+T^2\|f\|^2_{L^{1}(0,T; L^{2}(\Omega))}+T^2\|h\|^2_{L^{2}(\Omega)}\right).
\end{align*}
and
\begin{align*}
&\|p\|^{2}_{L^{2}(0,T; L^{2}(\Omega))}
+\|\nabla p\|^{2}_{L^{2}(0,T; \boldsymbol {L}^{2}(\Omega))}\nonumber\\
&\quad +\|\partial_{t}\boldsymbol{u}\|^{2}_{L^{2}(0,T; \boldsymbol{L}^{2}(D))}
+\|\nabla\boldsymbol{u}\|^{2}_{L^{2}(0,T; L^{2}(D)^{d \times d})}
+\|\nabla\cdot\boldsymbol{u}\|^{2}_{L^{2}(0,T; {L}^{2}(D))}\nonumber\\
&\leq C_{2}\left(T\|g\|^{2}_{L^{2}(\Omega)}
+T^3\|f\|^2_{L^{1}(0,T; L^{2}(\Omega))}+T^3\|h\|^2_{L^{2}(\Omega)}\right),
\end{align*}
\end{theorem}

\begin{proof}
 Let $ 0< s <T$ and define an auxiliary function
\begin{align*}
\Psi_{1}(x, t)&=\int_{t}^{s} p (x, \tau) {\rm d} \tau, \quad x \in \Omega, \quad
0 \leq t \leq s.\\
\boldsymbol{\Psi}_{2}(x, t)&=\int_{t}^{s}\partial_{\tau}\boldsymbol{u} (x, \tau)
{\rm d} \tau, \quad
x \in D, \quad 0 \leq t \leq s.
\end{align*}
 It is clear to note that
\begin{align}\label{s3-t7-1}
\Psi_{1} (x, s)=0, \quad\quad \partial_t \Psi_{1}(x, t) =-p (x, t),
\end{align}
and
\begin{align}\label{s3-t7-2}
\boldsymbol{\Psi}_{2}(x, s)=0, \quad\quad \partial_t \boldsymbol{\Psi}_{2}(x, t)
=-\partial_t \boldsymbol{u}  (x, t).
\end{align}
For any $\phi (x, t) \in L^2 (0, s; L^{2} (\Omega))$,
using integration by parts and \eqref{s3-t7-1}, we have
\begin{align}\label{s3-t7-3}
\int_{0}^{s}\phi(x, t)\Psi_{1}(x, t) {\rm d} t
&=\int_{0}^{s} \left(\phi(x, t)\int_{t}^{s} p (x, \tau) {\rm d} \tau  \right)
{\rm d} t\nonumber\\
&=\int_{0}^{s} \bigg[\left(\int_{0}^{t} \phi(x, \tau){\rm d}\tau\right)'
\left(\int_{t}^{s} p (x, \tau) {\rm d} \tau\right)  \bigg] {\rm d} t\nonumber\\
&=\bigg[\left(\int_{0}^{t} \phi(x, \tau){\rm d}\tau\right)
\left(\int_{t}^{s} p (x, \tau) {\rm d} \tau\right)
\bigg]\bigg|_{0}^{s}\nonumber\\
&\quad-\int_{0}^{s} \bigg[\left(\int_{0}^{t} \phi(x, \tau){\rm d}\tau\right)
\left(\int_{t}^{s} p(x, \tau) {\rm d} \tau\right)'  \bigg] {\rm d} t\nonumber\\
&=-\int_{0}^{s} \bigg[\left(\int_{0}^{t} \phi(x, \tau){\rm d}\tau\right)
\left(\int_{t}^{s} p(x, \tau) {\rm d} \tau\right)'  \bigg] {\rm d} t\nonumber\\
&=-\int_{0}^{s} \bigg[\left(\int_{0}^{t} \phi(x, \tau){\rm d}\tau\right)
\left(-p(x,t)\right) \bigg] {\rm d} t\nonumber\\
&= \int_{0}^{s}\left(\int_{0}^{t} \phi (x, \tau) {\rm d} \tau \right)
 p (x, t) {\rm d} t.
\end{align}

Taking the test function $q =\Psi_{1}$ in \eqref{vp} and integrating from
$t=0$ to $t= s$ yields
\begin{align} \label{s3-t7-4}
\int_{0}^{s}\left(\int_{\Omega}\frac{\beta}{c^{2}}(\partial_{t}^{2}p) \Psi_{1}
{\rm d}x\right){\rm d}t
&=-\int_{0}^{s}\left(\int_{\Omega}(M^{\frac{1}{2}}\nabla{p})\cdot(M^{\frac{1}{2}
}\nabla{\Psi_{1}}){\rm d}x\right){\rm d}t\nonumber\\
&\quad+\int_{0}^{s}\left(\int_{\partial D}\rho_{1}(\boldsymbol{n}_{D}\cdot
\partial_t^2\boldsymbol{u}) \Psi_{1}{\rm d} s\right){\rm d}t
+\int_{0}^{s}\left(\int_{\Omega}f\Psi_{1}{\rm d}x\right){\rm d}t\nonumber \\
&=-\int_{0}^{s}\left(\int_{\Omega}(M^{\frac{1}{2}}\nabla{p})\cdot(M^{\frac{1}{2}
}\nabla{\Psi_{1}}){\rm d}x\right){\rm d}t\nonumber\\
&\quad+\int_{0}^{s}\left(\int_{\partial D}\rho_{1}(\boldsymbol{n}_{D}\cdot
\partial_t\boldsymbol{u}) p {\rm d} s\right){\rm d}t
+\int_{0}^{s}\left(\int_{\Omega}f\Psi_{1}{\rm d}x\right){\rm d}t,
\end{align}

It follows from \eqref{s3-t7-1} that
\begin{align}\label{s3-t7-5}
&\int_{0}^{s}\left(\int_{\Omega}\frac{\beta}{c^{2}}(\partial_{t}^{2}p) \Psi_{1}
{\rm d}x\right){\rm d}t
=\int_{\Omega}\int_{0}^{s}\frac{\beta}{c^{2}}(\partial_{t}(\partial_{t}p\Psi_{1
})+p \partial_{t}p){\rm d}t{\rm d} x\nonumber\\
&=\int_{\Omega}\frac{\beta}{c^{2}}\left(\partial_{t}p\Psi_{1}\bigg|_{0}^{s}
+\frac{1}{2}|p|^{2}\bigg|_{0}^{s}\right){\rm d} x\nonumber\\
&=\frac{1}{2}\|\sqrt{\frac{\beta}{c^{2}}}p(\cdot,s)\|^{2}_{L^{2}(\Omega)}
-\frac{1}{2}\|\sqrt{\frac{\beta}{c^{2}}}g\|^{2}_{L^{2}(\Omega)}
-c^{-2}\beta\int_{\Omega}h(x)\Psi_{1}(x,0) {\rm d}x.
\end{align}

It follows from \eqref{s3-t7-4} and \eqref{s3-t7-5} that
\begin{align} \label{s3-t7-6}
&\frac{1}{2}\|\sqrt{\frac{\beta}{c^{2}}}p(\cdot,s)\|^{2}_{L^{2}(\Omega)}
+\int_{0}^{s}\left(\int_{\Omega}(M^{\frac{1}{2}}\nabla{p})\cdot(M^{\frac{1}{2}}
\nabla{\Psi_{1}}){\rm d}x\right){\rm d}t\nonumber\\
&=\int_{0}^{s}\left(\int_{\partial D}\rho_{1}(\boldsymbol{n}_{D}\cdot
\partial_t\boldsymbol{u}) p {\rm d} s\right){\rm d}t
+\int_{0}^{s}\left(\int_{\Omega}f(x,t)\Psi_{1}(x,t){\rm d}x\right){\rm
d}t\nonumber\\
&\quad+\frac{1}{2}\|\sqrt{\frac{\beta}{c^{2}}}g\|^{2}_{L^{2}(\Omega)}
+c^{-2}\beta\int_{\Omega}h(x)\Psi_{1}(x,0) {\rm d}x\nonumber\\
&=\frac{1}{2}\|\sqrt{\frac{\beta}{c^{2}}}p(\cdot,s)\|^{2}_{L^{2}(\Omega)}
+\frac{1}{2}\int_{\Omega}\bigg|\int_{0}^{s}M^{\frac{1}{2}}\nabla{p}(x,t){\rm
d}t\bigg|^{2}{\rm d}x.
\end{align}

Similarly, for any $\boldsymbol{\Phi} (x, t) \in L^2 (0, s; L^{2}
(\Omega)^{d})$,
using integration by parts and \eqref{s3-t7-2}, we have
\[
\int_{0}^{s}\boldsymbol{\Phi}(x, t)\cdot\boldsymbol{\Psi}_{2}(x, t) {\rm d} t
= \int_{0}^{s}\left(\int_{0}^{t} \boldsymbol{\Phi}(x, \tau) {\rm d} \tau \right)
\cdot \partial_{t}\boldsymbol{u} (x, t) {\rm d} t.
\]

We get from taking the test function $\boldsymbol{v} =\boldsymbol{\Psi}_{2}$ in
\eqref{vp1} and integrating from $t=0$ to $t= s$ that
\begin{align} \label{s3-t7-8}
&\int_{0}^{s}\left(\int_{D}\rho_{2}\partial_{t}^{2}(\partial_{t}\boldsymbol{u}
)\cdot\boldsymbol{\Psi}_{2}{\rm d}x\right){\rm d}t\nonumber\\
&=-\int_{0}^{s}\left(\int_{D}[(\mu\nabla\partial_{t}\boldsymbol{u}):(\nabla
\boldsymbol{\Psi}_{2})
+(\lambda+\mu)((\nabla\cdot\partial_{t}\boldsymbol{u})(\nabla\cdot\boldsymbol{
\Psi}_{2}))]{\rm d}x\right){\rm d}t\nonumber\\
&\quad-\int_{0}^{s}\left(\int_{\partial
D}(\partial_{t}p)(\boldsymbol{n}_{D}\cdot \boldsymbol{\Psi}_{2}){\rm d}
s\right){\rm d}t.
\end{align}

Using \eqref{s3-t7-2} and initial condition \eqref{iv2n}, we deduce
\begin{align}\label{s3-t7-9}
\int_{0}^{s}\left(\int_{D}\rho_{2}\partial_{t}^{2}(\partial_{t}\boldsymbol{u}
)\cdot\boldsymbol{\Psi}_{2}{\rm d}x\right){\rm d}t
&=\int_{D}\int_{0}^{s}\rho_{2}\left(\partial_{t}(\partial_{t}^{2}\boldsymbol{u}
\cdot\boldsymbol{\Psi}_{2}
+\partial_{t}^{2}\boldsymbol{u}\cdot\partial_{t}\boldsymbol{u})\right){\rm d} t
{\rm d} x\nonumber\\
&=\int_{D}\rho_{2}\left((\partial_{t}^{2}\boldsymbol{u}\cdot\boldsymbol{\Psi}_{2
})|_{0}^{s}+\frac{1}{2}|\partial_{t}\boldsymbol{u}|^{2}|_{0}^{s}\right){\rm d}
x\nonumber\\
&=\frac{\rho_{2}}{2}\|\partial_{t}\boldsymbol{u}(\cdot,s)\|^{2}_{\boldsymbol{L}^
{2}(D)}
\end{align}
and
\begin{align}\label{s3-t7-10}
\int_{0}^{s}\left(\int_{\partial D}(\partial_{t}p)(\boldsymbol{n}_{D}\cdot
\boldsymbol{\Psi}_{2}){\rm d} s\right){\rm d}t
&=\int_{\partial D}\int_{0}^{s}[\partial_{t}(p(\boldsymbol{n}_{D}\cdot
\boldsymbol{\Psi}_{2}))
+p(\boldsymbol{n}_{D}\cdot \partial_{t}\boldsymbol{u})]{\rm d}t{\rm
d}x\nonumber\\
&=\int_{\partial D}(p(\boldsymbol{n}_{D}\cdot
\boldsymbol{\Psi}_{2}))|_{0}^{s}{\rm d}s
+\int_{0}^{s}\int_{\partial D}p(\boldsymbol{n}_{D}\cdot
\partial_{t}\boldsymbol{u}){\rm d}s{\rm d}t\nonumber\\
&=\int_{0}^{s}\int_{\partial D}p(\boldsymbol{n}_{D}\cdot
\partial_{t}\boldsymbol{u}){\rm d}s{\rm d}t.
\end{align}
Using \eqref{s3-t7-8}, \eqref{s3-t7-9} and \eqref{s3-t7-10} yields
\begin{align} \label{s3-t7-11}
&\frac{\rho_{2}}{2}\|\partial_{t}\boldsymbol{u}(\cdot,s)\|^{2}_{\boldsymbol{L}^{
2}(D)}
+\int_{0}^{s}\left(\int_{D}[(\mu\nabla\partial_{t}\boldsymbol{u}):(\nabla
\boldsymbol{\Psi}_{2})
+(\lambda+\mu)(\nabla\cdot\partial_{t}\boldsymbol{u})(\nabla\cdot\boldsymbol{
\Psi}_{2})]{\rm d}x\right){\rm d}t\nonumber\\
&=\frac{\rho_{2}}{2}\|\partial_{t}\boldsymbol{u}(\cdot,s)\|^{2}_{\boldsymbol{L}^
{2}(D)} +\frac{1}{2}\left(\mu\bigg\|\int_{0}^{s}\nabla(\partial_{t}
\boldsymbol{u} (\cdot,t)){\rm d} t\bigg\|_{L^{2}(D)^{d \times d}}^{2}
+(\lambda+\mu)\int_{D}\bigg|\int_{0}^{s}\nabla\cdot(\partial_{t}\boldsymbol{u}
(\cdot,t)){\rm d} t\bigg|^{2}{\rm d}x\right)\nonumber\\
&=-\int_{0}^{s}\int_{\partial D}p(\boldsymbol{n}_{D}\cdot
\partial_{t}\boldsymbol{u}){\rm d}s{\rm d}t.
\end{align}
Multiplying \eqref{s3-t7-6} by $\rho_{1}$ and then adding it to
\eqref{s3-t7-11}, we obtain
\begin{align} \label{s3-t7-12}
&\frac{1}{2}\|\sqrt{\frac{\beta}{c^{2}}}p(\cdot,s)\|^{2}_{L^{2}(\Omega)}
+\frac{1}{2}\int_{\Omega}\bigg|\int_{0}^{s}M^{\frac{1}{2}}\nabla{p}(\boldsymbol{
\cdot},t){\rm d}t\bigg|^{2}{\rm d}x\nonumber\\
&+\frac{\rho_{1}\rho_{2}}{2}\|\partial_{t}\boldsymbol{u}(\cdot,s)\|^{2}_{
\boldsymbol{L}^{2}(D)}
+\frac{\rho_{1}}{2}\left(\mu\bigg\|\int_{0}^{s}\nabla(\partial_{t}
\boldsymbol{u}(\cdot,t)){\rm d} t\bigg\|_{L^{2}(D)^{d \times d}}^{2}
+(\lambda+\mu)\int_{D}\bigg|\int_{0}^{s}\nabla\cdot(\partial_{t}\boldsymbol{u}
(\cdot,t)){\rm d} t\bigg|^{2}{\rm d}x\right)\nonumber\\
&=\int_{0}^{s}\left(\int_{\Omega}f(x,t)\Psi_{1}(x,t){\rm d}x\right){\rm d}t
+\frac{1}{2}\|\sqrt{\frac{\beta}{c^{2}}}g\|^{2}_{L^{2}(\Omega)}
+c^{-2}\beta\int_{\Omega}h(x)\Psi_{1}(x,0) {\rm d}x.
\end{align}

Next, we  estimate the two terms on the left-hand side of
\eqref{s3-t7-12} separately. It follows from the Cauchy--Schwarz inequality that
\begin{align}\label{s3-t7-13}
c^{-2}\beta\int_{\Omega}h(x)\Psi_{1}(x,0) {\rm d}x
&=c^{-2}\beta\int_{\Omega}h(x)\left(\int_{0}^{s}p(x,t){\rm d} t\right){\rm d}
x\nonumber\\
&=c^{-2}\beta\int_{0}^{s}\left(\int_{\Omega}h(x)p(x,t){\rm d} x\right){\rm d}
t\nonumber\\
&\leq C\left(\|h\|_{L^{2}(\Omega)}\right)
\int_{0}^{s}\|p(\cdot,t)\|_{L^{2}(\Omega)}{\rm d} t.
\end{align}
For $0 \leq t \leq s \leq T,$ we have from \eqref{s3-t7-3} that
\begin{align}\label{s3-t7-14}
\int_{0}^{s}\left(\int_{\Omega} f(x,t) \Psi_{1}(x,t){\rm d} x\right){\rm d} t
&=\int_{\Omega}\left(\int_{0}^{s} \left(\int_{0}^{t}f(x,\tau){\rm d}\tau\right)
 p(x,t){\rm d} t\right){\rm d} x\nonumber\\
&\leq \int_{0}^{s} \int_{0}^{t}\|f(\cdot,\tau)\|_{L^{2}(\Omega)}
\|p(\cdot,t)\|_{L^{2}(\Omega)}{\rm d}\tau {\rm d} t\nonumber\\
&\leq \left(\int_{0}^{s}\|f(\cdot,t)\|_{L^{2}(\Omega)}  {\rm d} t\right)
\left(\int_{0}^{s}\|p(\cdot,t)\|_{L^{2}(\Omega)}{\rm d}t \right).
\end{align}
Substituting \eqref{s3-t7-13}-\eqref{s3-t7-14} into \eqref{s3-t7-12}, we have
for any $s \in [0, T]$ that
\begin{align} \label{s3-t7-15}
&\frac{1}{2}\|\sqrt{\frac{\beta}{c^{2}}}p(\cdot,s)\|^{2}_{L^{2}(\Omega)}
+\frac{1}{2}\int_{\Omega}\bigg|\int_{0}^{s}M^{\frac{1}{2}}\nabla{p}(\boldsymbol{
\cdot},t){\rm d}t\bigg|^{2}{\rm d}x\nonumber\\
&+\frac{\rho_{1}\rho_{2}}{2}\|\partial_{t}\boldsymbol{u}(\cdot,s)\|^{2}_{L^{2}
(D)^{2}} +\frac{\rho_{1}}{2}\left(\mu\bigg\|\int_{0}^{s}\nabla(\partial_{t}
\boldsymbol{u}(\cdot,t)){\rm d} t\bigg\|_{L^{2}(D)^{d \times d}}^{2}
+(\lambda+\mu)\int_{D}\bigg|\int_{0}^{s}\nabla\cdot(\partial_{t}\boldsymbol{u}
(\cdot,t)){\rm d} t\bigg|^{2}{\rm d}x\right)\nonumber\\
&\leq\frac{\beta}{2c^{2}}\|g\|^{2}_{L^{2}(\Omega)}
+\left(\int_{0}^{s}\|f(\cdot,t)\|_{L^{2}(\Omega)}{\rm d}t+C\|h\|_{L^{2}(\Omega)}\right)
\int_{0}^{s}\|p(\cdot,t)\|_{L^{2}(\Omega)}{\rm d} t.
\end{align}

Taking the $L^{\infty}$- norm with respect to $s$ on both sides of \eqref{s3-t7-15} yields
\begin{align*}
&\|p\|^{2}_{L^{\infty}(0,T; L^{2}(\Omega))}
+\|\nabla p\|^{2}_{L^{\infty}(0,T; \boldsymbol{L}^{2}(\Omega))}\nonumber\\
&\quad +\|\partial_{t}\boldsymbol{u}\|^{2}_{L^{\infty}(0,T;
\boldsymbol{L}^{2}(D))}
+\|\nabla\boldsymbol{u}\|^{2}_{L^{\infty}(0,T; L^{2}(D)^{d \times d})}
+\|\nabla\cdot\boldsymbol{u}\|^{2}_{L^{\infty}(0,T; {L}^{2}(D))}\nonumber\\
&\leq C_{1}\|g\|^{2}_{L^{2}(\Omega)}
+C_{2}T\left(\|f\|_{L^{1}(0,T; L^{2}(\Omega))}+\|h\|_{L^{2}(\Omega)}\right)
\|p\|_{L^{\infty}(0,T; L^{2}(\Omega))}.
\end{align*}
Applying the Young inequality yields
\begin{align*}
&\|p\|^{2}_{L^{\infty}(0,T; L^{2}(\Omega))}
+\|\nabla p\|^{2}_{L^{\infty}(0,T; \boldsymbol{L}^{2}(\Omega))}\nonumber\\
&\quad +\|\partial_{t}\boldsymbol{u}\|^{2}_{L^{\infty}(0,T;
\boldsymbol{L}^{2}(D))}
+\|\nabla\boldsymbol{u}\|^{2}_{L^{\infty}(0,T; L^{2}(D)^{d\times d})}
+\|\nabla\cdot\boldsymbol{u}\|^{2}_{L^{\infty}(0,T; {L}^{2}(D))}\nonumber\\
&\leq C_{1}\left(\|g\|^{2}_{L^{2}(\Omega)}
+T^2\|f\|^2_{L^{1}(0,T; L^{2}(\Omega))}+T^2\|h\|^2_{L^{2}(\Omega)}\right).
\end{align*}
Integrating \eqref{s3-t7-15} with respect to $s$ from $0$ to $T$ and using the
Cauchy-Schwarz inequality and the Young inequality, we can get
\begin{align*}
&\|p\|^{2}_{L^{2}(0,T; L^{2}(\Omega))}
+\|\nabla p\|^{2}_{L^{2}(0,T; \boldsymbol{L}^{2}(\Omega))}\nonumber\\
&\quad+\|\partial_{t}\boldsymbol{u}\|^{2}_{L^{2}(0,T; \boldsymbol{L}^{2}(D))}
+\|\nabla\boldsymbol{u}\|^{2}_{L^{2}(0,T; L^{2}(D)^{d \times d})}
+\|\nabla\cdot\boldsymbol{u}\|^{2}_{L^{2}(0,T; {L}^{2}(D))}\nonumber\\
&\leq C_{2}\left(T\|g\|^{2}_{L^{2}(\Omega)}
+T^3\|f\|^2_{L^{1}(0,T; L^{2}(\Omega))}+T^3\|h\|^2_{L^{2}(\Omega)}\right),
\end{align*}
which completes the proof.
\end{proof}

\section{Conclusion}\label{cl}

In this paper, we have studied the two- and three-dimensional acoustic-elastic
wave scattering problem on a finite time interval. The acoustic and elastic wave
equations are coupled on the surface of the elastic obstacle. We propose the
compressed coordinate transformation to reduce equivalently the scattering
problem into an initial-boundary value problem in a bounded domain. The reduced
problem is proved to have a unique weak solution by using the Galerkin method.
An a priori estimate with explicit time dependence is also obtained for the
acoustic pressure and elastic displacement of the time-domain variational
problem. We believe that the method of compressed coordinate transformation can
be applied to many other time-domain scattering problems imposed in open
domains. The model problem is suitable for numerical simulations. We hope
to report the work on the numerical analysis and computation elsewhere in the
future.

\appendix

\section{change of variables in two dimensions}

Let $\boldsymbol x=(x, y)\in\mathbb R^2$ and $\rho=|\boldsymbol x|$. The polar
coordinates $(\rho,\theta)$ are related to the Cartesian coordinates $(x, y)$
by $x=\rho\cos\theta, y=\rho\sin\theta$. The local orthonormal basis is
\[
\boldsymbol e_\rho=(\cos\theta, \sin\theta)^\top,\quad \boldsymbol
e_\theta=(-\sin\theta, \cos\theta)^\top.
\]
Denote by $\nabla_\rho$ and $\nabla_\rho\cdot$ the gradient operator and the
divergence operator in the old coordinates $(\rho, \theta)$, respectively. We
study the two-dimensional acoustic wave equation:
\begin{equation}\label{Ane}
\frac{1}{c^2}\partial_t^2 u(\rho, \theta, t)-\Delta_\rho
u(\rho, \theta, t)=0\quad\text{in}~\mathbb R^2,~ t>0,
\end{equation}
where $\Delta_\rho$ is the Laplace operator and $c>0$ is the wave speed.

Consider the change of variables $\rho=\zeta(r)$, where $\zeta$ is a
smooth and invertible function. Denote by $\nabla_r$ and $\nabla_r\cdot$
the gradient operator and the divergence operator in the new coordinates $(r,
\theta)$, respectively.

\begin{lemma}\label{A1}
 Let $v(r,\theta,t)=u(\rho,\theta,t)|_{\rho=\zeta(r)}$ be a differentiable
scalar function, then
\[
\nabla_{\rho}u(\rho,\theta,t)|_{\rho=\zeta(r)}=Q
\begin{bmatrix}
\frac{1}{\zeta'(r)} & 0\\
0 & \frac{r}{\zeta(r)}
\end{bmatrix}
Q^{\top}\nabla_{r} v(r,\theta,t),
\]
where $R$ is an orthonormal matrix given by
\[
 Q(\theta)=\begin{bmatrix}
    \cos\theta & -\sin\theta\\
    \sin\theta & \cos\theta
   \end{bmatrix}.
\]
\end{lemma}

\begin{proof}
It follows from the straightforward calculations that
 \begin{align*}
\nabla_{\rho} u|_{\rho=\zeta(r)}&=
\partial_{\rho}u|_{\rho=\zeta(r)}\boldsymbol{e}_{\rho}
+\frac{1}{\rho}\partial_{\theta}u|_{\rho=\zeta(r)}\boldsymbol{e}_{
\theta}\\
&=\frac{1}{\zeta'}\partial_{r}v\boldsymbol{e}_r+\frac{1}{\zeta}
\partial_{\theta}v\boldsymbol{e}_{\theta}\\	
&= \begin{bmatrix}
\cos\theta & -\sin\theta\\
\sin\theta & \cos\theta
\end{bmatrix}
\begin{bmatrix}
\frac{1}{\zeta'} & 0\\
0 & \frac{r}{\zeta}
\end{bmatrix}
\begin{bmatrix}
\partial_r v\\
\frac{1}{r}\partial_{\theta}v
\end{bmatrix}\\
&= Q\begin{bmatrix}
\frac{1}{\zeta'} & 0\\
0 & \frac{r}{\zeta}
\end{bmatrix}
Q^\top Q
\begin{bmatrix}
\partial_r v\\
\frac{1}{r}\partial_{\theta}v
\end{bmatrix}\\
&=Q\begin{bmatrix}
\frac{1}{\zeta'} & 0\\
0 & \frac{r}{\zeta}
\end{bmatrix}
Q^{\top}\nabla_{r}v,
\end{align*}
which completes the proof.
\end{proof}

\begin{lemma}\label{A2}
Let $\boldsymbol v(r,\theta,t)=\boldsymbol
u(\rho,\theta,t)|_{\rho=\zeta(r)}$ be a differentiable vector function, then
\[
\nabla_\rho\cdot\boldsymbol u(\rho,
\theta, t)|_{\rho=\zeta(r)}=\beta^{-1}(r)\nabla_{r}\cdot\left(K(r,
\theta)\boldsymbol{v}(r,\theta, t) \right),
\]
where
\[
 \beta(r)=\frac{\zeta(r)\zeta'(r)}{r},\quad K(r,\theta)=Q
\begin{bmatrix}
\frac{\zeta(r)}{r} & 0\\
0 & \zeta'(r)
\end{bmatrix} Q^{\top}.
\]
\end{lemma}

\begin{proof}
Let $\boldsymbol u=u_\rho\boldsymbol e_\rho+u_s\boldsymbol e_\theta$ and
$\boldsymbol v=v_r\boldsymbol e_r+v_\theta\boldsymbol e_\theta$. A simple
calculation yields that
\begin{align*}
 \nabla_\rho\cdot\boldsymbol u(\rho, \theta,
t)|_{\rho=\zeta(r)}&=\frac{1}{\rho}\frac{\partial}{\partial\rho}\left(\rho
u_{\rho}\right)+\frac{1}{\rho}\partial_{\theta}(u_{\theta})\\
&=\frac{1}{\zeta \zeta'}\partial_{r} \left(\zeta
v_r\right)+\frac{1}{\zeta}\partial_{\theta}(v_{\theta})\\
&=\frac{r}{\zeta \zeta'}\left[\frac{1}{r}\partial_{r}\left(r\frac{\zeta}{r}
v_{r}\right)+\frac{1}{r} \partial_{\theta}\left(\zeta'
v_{\theta}\right)\right]\\
&=\beta^{-1}\nabla_{r}\cdot \left(\frac{\zeta}{r}v_{r}\boldsymbol{e}_r 	
+\zeta' v_{\theta}\boldsymbol{e}_{\theta}\right)\\
&=\beta^{-1}\nabla_{r}\cdot\left(Q
\begin{bmatrix}
\frac{\zeta}{r} & 0\\
0 & \zeta'
\end{bmatrix}
Q^\top Q
\begin{bmatrix}
v_{r}\\
v_{\theta}
\end{bmatrix}
\right)\\
&=\beta^{-1}\nabla_{r}\cdot\left(K\boldsymbol{v}
\right),
\end{align*}
which completes the proof.
\end{proof}

\begin{lemma}\label{A3}
Let $v(r,\theta,t)= u(\rho,\theta,t)|_{\rho=\zeta(r)}$ be a differentiable
function, then
\[
\Delta_\rho u(\rho,\theta,
t)|_{\rho=\zeta(r)}=\beta^{-1}(r)\nabla_r\cdot
\left(M(r,\theta)\nabla_r v(r,\theta,t)\right),
\]
where
\[
 M(r,\theta)=Q\begin{bmatrix}
     \frac{\zeta(r)}{r\zeta'(r)} & 0\\
     0 & \frac{r\zeta'(r)}{\zeta(r)}
    \end{bmatrix}Q^\top.
\]

\end{lemma}

\begin{proof}
It is easy to note that
\[
\Delta_\rho u|_{\rho=\zeta(r)}=\nabla_{\rho}\cdot(\nabla_\rho
u)|_{\rho=\zeta(r)}.
\]
Using similar steps of the changes of variables in the proofs for Lemmas
\ref{A1}--\ref{A2}, we have
\begin{align*}
 \nabla_{\rho}\cdot(\nabla_\rho u)
|_{\rho=\zeta(r)}&= \beta^{-1}\nabla_r\cdot\left(K Q\begin{bmatrix}
                                  \frac{1}{\zeta'} & 0\\
                                 0 & \frac{r}{\zeta}
                             \end{bmatrix}Q^\top\nabla_r v \right)\\
&=\beta^{-1}\nabla_r\cdot\left(Q\begin{bmatrix}
                                  \frac{\zeta}{r\zeta'} & 0\\
                                 0 & \frac{r\zeta'}{\zeta}
                             \end{bmatrix}Q^\top\nabla_r v\right)\\
                &=\beta^{-1}\nabla_r\cdot\left(M\nabla_r v
\right),
\end{align*}
which completes the proof.
\end{proof}

\begin{theorem}\label{A4}
 In the new coordinates $(r,\theta)$, the acoustic wave equation \eqref{Ane}
becomes
\[
 \frac{\beta(r)}{c^2}\partial_t^2 v(r, \theta,
t)-\nabla_r\cdot(M(r,\theta)\nabla_r v(r,\theta,t))=0.
\]
\end{theorem}

\begin{proof}
 Using Lemma \ref{A2}--Lemma \ref{A3}, we from \eqref{Ane} that
\[
0=\Big(\frac{1}{c^2}\partial_t^2 u-\Delta u\Big)\Big|_{\rho=\zeta(r)}
=\frac{1}{c^2}\partial_t^2 v-\beta^{-1}\nabla_r\cdot
\left(M\nabla_r  v\right).
\]
The proof is completed by multiplying $\beta$ on the above equation.
\end{proof}

\section{change of variables in three dimensions}

Let $\boldsymbol x=(x, y, z)\in\mathbb R^3$ and $\rho=|\boldsymbol x|$. The
spherical coordinates $(\rho, \theta, \varphi)$ are related to the Cartesian
coordinates $(x, y, z)$ by $x=\rho\sin\theta\cos\varphi,
y=\rho\sin\theta\sin\varphi, z=\rho\cos\theta$. The local orthonormal basis is
\begin{align*}
\boldsymbol e_\rho&=(\sin\theta\cos\varphi, \sin\theta\sin\varphi,
\cos\theta)^\top,\\
\boldsymbol e_\theta&=(\cos\theta\cos\varphi,
\cos\theta\sin\varphi, -\sin\theta)^\top,\\
\boldsymbol e_\varphi&=(-\sin\varphi, \cos\varphi, 0)^\top.
\end{align*}
Again, denote by $\nabla_\rho$ and $\nabla_\rho\cdot$ the gradient operator and
the divergence operator in the old coordinates $(\rho, \theta, \varphi)$,
respectively. In this section, we present parallel results for the
three-dimensional acoustic wave equation:
\begin{equation}\label{Bne}
\frac{1}{c^2}\partial_t^2
u(\rho,\theta,\varphi,t)-\Delta_\rho u(\rho,\theta,\varphi,
t)=0\quad\text{in}~\mathbb R^3,~ t>0,
\end{equation}
where $\Delta_\rho$ is the Laplace operator and $c>0$ is the wave speed.

Consider the change of variables $\rho=\zeta(r)$, where $\zeta$ is a
smooth and invertible function. Denote by $\nabla_r$ and $\nabla_r\cdot$
the gradient operator and the divergence operator in the new coordinates $(r,
\theta,\varphi)$, respectively.

\begin{lemma}\label{B1}
 Let $v(r,\theta,\varphi,t)=u(\rho,\theta,\varphi,t)|_{\rho=\zeta(r)}$ be a
differentiable scalar function, then
\[
\nabla_{\rho}u(\rho,\theta,\varphi,t)|_{\rho=\zeta(r)}=Q
\begin{bmatrix}
\frac{1}{\zeta'(r)} & 0 & 0\\
0 & \frac{r}{\zeta(r)} & 0\\
0 & 0 & \frac{r}{\zeta(r)}
\end{bmatrix}
Q^{\top}\nabla_{r} v(r,\theta,\varphi,t),
\]
where $R$ is an orthonormal matrix given by
\[
 Q(\theta,\varphi)=\begin{bmatrix}
   \sin\theta\cos\varphi & \cos\theta\cos\varphi & -\sin\varphi\\
\sin\theta\sin\varphi & \cos\theta\sin\varphi & \cos\varphi\\
\cos\theta & -\sin\theta & 0
   \end{bmatrix}.
\]
\end{lemma}

\begin{proof}
It follows from the straightforward calculations that
 \begin{align*}
\nabla_{\rho} u|_{\rho=\zeta(r)}&=
\partial_{\rho}u|_{\rho=\zeta(r)}\boldsymbol{e}_{\rho}
+\frac{1}{\rho}\partial_{\theta}u|_{\rho=\zeta(r)}\boldsymbol{e}_{
\theta}+\frac{1}{\rho\sin\theta}
\partial_\varphi u|_{\rho=\zeta(r)}\boldsymbol
e_\varphi\\
&=\frac{1}{\zeta'}\partial_{r}v\boldsymbol{e}_r+\frac{1}{\zeta}
\partial_{\theta}v\boldsymbol{e}_{\theta}+\frac{1}{\zeta\sin\theta}
\partial_\varphi v\boldsymbol e_\varphi\\	
&= \begin{bmatrix}
\sin\theta\cos\varphi & \cos\theta\cos\varphi & -\sin\varphi\\
\sin\theta\sin\varphi & \cos\theta\sin\varphi & \cos\varphi\\
\cos\theta & -\sin\theta & 0
\end{bmatrix}
\begin{bmatrix}
\frac{1}{\zeta'} & 0 & 0\\
0 & \frac{r}{\zeta} & 0\\
0 & 0 & \frac{r}{\zeta}
\end{bmatrix}
\begin{bmatrix}
\partial_r v\\
\frac{1}{r}\partial_{\theta}v\\
\frac{1}{r\sin\theta}\partial_\varphi v
\end{bmatrix}\\
&= Q\begin{bmatrix}
\frac{1}{\zeta'} & 0 & 0\\
0 & \frac{r}{\zeta} & 0\\
0 & 0 & \frac{r}{\zeta}
\end{bmatrix}
Q^\top Q
\begin{bmatrix}
\partial_r v\\
\frac{1}{r}\partial_{\theta}v\\
\frac{1}{r\sin\theta}\partial_\varphi v
\end{bmatrix}\\
&=Q\begin{bmatrix}
\frac{1}{\zeta'} & 0 & 0\\
0 & \frac{r}{\zeta} & 0\\
0 & 0 & \frac{r}{\zeta}
\end{bmatrix}
Q^{\top}\nabla_{r}v,
\end{align*}
which completes the proof.
\end{proof}

\begin{lemma}\label{B2}
Let $\boldsymbol v(r,\theta,\varphi,t)=\boldsymbol
u(\rho,\theta,\varphi,t)|_{\rho=\zeta(r)}$ be a differentiable vector function,
then
\[
\nabla_\rho\cdot\boldsymbol u(\rho,
\theta,\varphi,t)|_{\rho=\zeta(r)}=\beta^{-1}(r)\nabla_{r}\cdot\left(K(r,
\theta,\varphi)\boldsymbol{v}(r,\theta,\varphi,t) \right),
\]
where
\[
 \beta(r)=\frac{\zeta^2(r)\zeta'(r)}{r^2},\quad K(r,\theta,\varphi)=Q
\begin{bmatrix}
\frac{\zeta^2(r)}{r^2} & 0 & 0\\
0 & \frac{\zeta(r)\zeta'(r)}{r} & 0\\
0 & 0 & \frac{\zeta(r)\zeta'(r)}{r}
\end{bmatrix} Q^{\top}.
\]
\end{lemma}

\begin{proof}
Let $\boldsymbol u=u_\rho\boldsymbol e_\rho+u_\theta\boldsymbol
e_\theta+u_\varphi\boldsymbol e_\varphi$ and
$\boldsymbol v=v_r\boldsymbol e_r+v_\theta\boldsymbol
e_\theta+v_\varphi\boldsymbol e_\varphi$. A simple
calculation yields that
\begin{align*}
 \nabla_\rho\cdot\boldsymbol u(\rho, \theta,\varphi,
t)|_{\rho=\zeta(r)}&=\left(\frac{1}{\rho^2}\partial_\rho\left(\rho^2
u_{\rho}\right)+\frac{1}{\rho\sin\theta}\partial_{\theta}(\sin\theta u_{\theta}
)+\frac{1}{\rho\sin\theta}\partial_\varphi(u_\varphi)\right)\bigg|_{
\rho=\zeta(r)}\\
&=\frac{1}{\zeta^2 \zeta'}\partial_{r} \left(\zeta^2
v_r\right)+\frac{1}{\zeta\sin\theta}\partial_{\theta}(\sin\theta
v_{\theta})+\frac{1}{\zeta\sin\theta}\partial_\varphi(v_\varphi)\\
&=\frac{r^2}{\zeta^2
\zeta'}\left[\frac{1}{r^2}\partial_{r}\left(r^2\frac{\zeta^2 }{r^2}
v_{r}\right)+\frac{1}{r\sin\theta} \partial_{\theta}\left(\frac{\zeta\zeta'}{r}
\sin\theta v_{\theta}\right)+\frac{1}{r\sin\theta}\partial_\varphi\left(\frac{
\zeta\zeta'}{r} v_\varphi\right)\right]\\
&=\beta^{-1}\nabla_{r}\cdot \left(\frac{\zeta^2}{r^2}v_{r}\boldsymbol{e}_r 	
+\frac{\zeta\zeta'}{r} v_{\theta}\boldsymbol{e}_{\theta}+\frac{\zeta\zeta'}{r}
v_\varphi\boldsymbol e_\varphi\right)\\
&=\beta^{-1}\nabla_{r}\cdot\left(Q
\begin{bmatrix}
\frac{\zeta^2}{r^2} & 0 & 0\\
0 & \frac{\zeta\zeta'}{r} & 0\\
0 & 0 & \frac{\zeta\zeta'}{r}
\end{bmatrix}
Q^\top Q
\begin{bmatrix}
v_{r}\\
v_{\theta}\\
v_\varphi
\end{bmatrix}
\right)\\
&=\beta^{-1}\nabla_{r}\cdot\left(K\boldsymbol{v}
\right),
\end{align*}
which completes the proof.
\end{proof}

\begin{lemma}\label{B3}
Let $ v(r,\theta,\varphi,t)= u(\rho,\theta,\varphi,t)|_{\rho=\zeta(r)}$ be a
differentiable function, then
\[
\Delta_\rho u(\rho,\theta,\varphi,
t)|_{\rho=\zeta(r)}=\beta^{-1}(r)\nabla_r\cdot
\left(M(r,\theta,\varphi)\nabla_r v(r,\theta,\varphi,t)\right),
\]
where
\[
 M(r,\theta,\varphi)=Q\begin{bmatrix}
     \frac{\zeta^2(r)}{r^2\zeta'(r)} & 0 & 0\\
     0 & \zeta'(r) & 0)\\
     0 & 0 & \zeta'(r)
    \end{bmatrix}Q^\top.
\]
\end{lemma}

\begin{proof}
It is easy to note that
\[
 \Delta_\rho u|_{\rho=\zeta(r)}=
\nabla_{\rho}\cdot(\nabla_\rho u) |_{\rho=\zeta(r)}.
\]
Using similar steps of the change of variables for the proofs of Lemmas
\ref{B1}--\ref{B2}, we have
\begin{align*}
 \nabla_{\rho}\cdot(\nabla_\rho u)
|_{\rho=\zeta(r)}&=\beta^{-1}\nabla_r\cdot\left(KQ\begin{bmatrix}
                                  \frac{1}{\zeta'} & 0 & 0\\
                                 0 & \frac{r}{\zeta} & 0\\
                                 0 & 0 & \frac{r}{\zeta}
                             \end{bmatrix}Q^\top\nabla_r v\right)\\
&=\beta^{-1}\nabla_r\cdot\left(Q\begin{bmatrix}
                                  \frac{\zeta^2}{r^2 \zeta'} & 0 & 0\\
                                 0 & \zeta' & 0\\
                                 0 & 0 & \zeta'
                             \end{bmatrix}Q^\top\nabla_r v\right)\\
                &=\beta^{-1}\nabla_r\cdot\left(M\nabla_r v
\right),
\end{align*}
which completes the proof.
\end{proof}

\begin{theorem}\label{B4}
 In the new coordinates $(r,\theta,\varphi)$, the acoustic wave equation
\eqref{Bne} becomes
\begin{align*}
\frac{\beta(r)}{c^2}\partial_t^2
v(r,\theta,\varphi,t)-\nabla_r\cdot(M(r,\theta,\varphi)\nabla_r
v(r,\theta,\varphi,t))=0.
 \end{align*}
\end{theorem}

\begin{proof}
 Using Lemma \ref{B2}--Lemma \ref{B3}, we from \eqref{Bne} that
 \begin{align*}
  0=\Big(\frac{1}{c^2}\partial_t^2
u-\Delta_\rho u\Big)\Big|_{\rho=\zeta(r)}
=\frac{1}{c^2}\partial_t^2 v-\beta^{-1}\nabla_r\cdot
\left(M\nabla_r v\right).
 \end{align*}
 The proof is completed by multiplying $\beta$ on the above equation.
\end{proof}

\end{document}